\documentclass[12pt,a4paper]{article}
\usepackage[a4paper,top=2.5cm,bottom=2.5cm,left=2.5cm,right=2.5cm]{geometry}
\usepackage{amsmath}
 \usepackage{graphicx}
\usepackage{amsthm}
\usepackage{color}
\usepackage{fancyhdr}
\usepackage{pdfpages}
\usepackage{txfonts}
\usepackage{amssymb}
\usepackage{paralist}
\usepackage{extarrows}
\usepackage{setspace}
\usepackage{braket} %脚注缩进的宏包
\usepackage[colorlinks=true,citecolor=blue,linkcolor=blue,anchorcolor=blue,urlcolor=blue]{hyperref}
\allowdisplaybreaks[4]
\numberwithin{equation}{section}
\newtheorem{theorem}{Theorem}[section]
\newtheorem{lemma}{Lemma} [section]

\begin{document}
\title{Variational problems related to self-similar solutions of Hardy-Sobolev heat equation in $\mathbb{R}^N$
 \footnote{This study was supported by the Key Program of the National Natural Science Foundation of China (Grant No. 12231016.)
 and the Program of the National Natural Science Foundation of China (Grant No. 12571117.)}}
\author{Fei Fang \\  \footnotesize  \emph{School of Mathematics and Statistics,%
Beijing Technology and Business University, Beijing 100048, China}
\\Zhong Tan \footnote{Corresponding author.  Email: tanloy2026@126.com} \\  \footnotesize  \emph{School of Mathematical Science, Xiamen University, Xiamen, 361005, China}}
\maketitle
\noindent \textbf{\textbf{Abstract:}}
In this paper, we apply a self-similar transformation to convert the parabolic equation with a Sobolev-Hardy term
  \begin{align*}
u_t-\Delta u= \frac{|u|^{q-2}u}{\left|x\right|^s}  & \text { in } \mathbb{R}^N \times(0, \infty), 
\end{align*}
into the following elliptic equation
\begin{equation*}
-\Delta v-\frac{1}{2} y \cdot \nabla v=\alpha v+ \frac{|v|^{q-2} v}{|y|^s},
\end{equation*}
where  $2 < q \leq 2^*(s)=\frac{2 N-2 s}{N-2}, 0 \leq s < 2, \alpha=\frac{2-s}{2q-4}$.
For this equation, we establish the
weighted Hardy inequality and Sobolev inequality. Furthermore, by virtue of the variational methods,
we obtain infinitely many solutions in the subcritical case,
 and prove the existence of solutions in the critical case.
We also apply the Pohozaev identity to establish the nonexistence of solutions under certain conditions.



\section{Introduction}
This paper  is concerned with  the following classical semilinear parabolic equation associated with  Sobolev-Hardy term in $\mathbb{R}^N$ :
\begin{equation}\label{e01}
u_t-\Delta u= \frac{|u|^{q-2}u}{\left|x\right|^s} ,  \text { in } \mathbb{R}^N \times(0, \infty),
\end{equation}
where  $2 < q \leq 2^*(s) \equiv \frac{2 N-2 s}{N-2}, 0 \leq s < 2$.
Semilinear parabolic partial differential equations with singular potentials constitute a core research branch of
modern nonlinear analysis and evolutionary equation theory. Compared with classical parabolic equations with smooth
 coefficients, such equations feature spatially singular terms, which can effectively characterize complex diffusion,
  concentration and singular evolution phenomena in physics, mechanics and engineering practice, including anomalous
  heat conduction, phase transition dynamics, material singularity response and population evolution with spatial
  singularity constraints (see \cite{dr3,dr4}.  In \cite{dr5}, the author established the necessary and sufficient
  conditions for initial data lying below or at the ground state, which lead to a complete dichotomy of the dynamical
   behavior of solutions: the solution either exists globally in time with its energy decaying to zero as time evolves,
   or undergoes blow-up in finite or infinite time.
In \cite{dr9}, for spatial dimension \(N\geq5\), by virtue of the Fourier splitting method,
the author derived the precise decay rate of dissipative solutions and demonstrated that the critical norm of such
dissipative solutions vanish as time tends to infinity.
  For further research progress on this equation, we refer the
   readers to  \cite{dr8,dr7,dr9,dr6} and the references therein.

The self-similar transformation is defined as follow:
$$
u(x, t)=t^{-\alpha} v(y), \quad y=\frac{x}{\sqrt{t}},
$$
where,
$$
t>0, \alpha=\frac{2-s}{2q-4}.
$$
Through this transformation and the detailed calculation, problem \eqref{e01} can be transformed into the following problem
\begin{equation}\label{e02}
-\Delta v-\frac{1}{2} y \cdot \nabla v=\alpha v+ \frac{|v|^{q-2} v}{|y|^s}.
\end{equation}
Set
$
K(y):=e^{|y|^2 / 4}.
$
Then we let $$L v:=-\Delta v-\frac{1}{2} y \cdot \nabla v=-\frac{1}{K} \nabla \cdot(K \nabla v).$$
Hence,  problem \eqref{e02} is equivalent to
\begin{equation*}
Lv=\alpha v+ \frac{|v|^{q-2} v}{|y|^s}.
\end{equation*}
When \(s=0\), for problem \eqref{e01}, Tan \cite{tan} adopted the Moser iteration technique and potential well method to establish the existence of classical global solutions for initial data with low energy. Meanwhile, he further investigated the asymptotic behaviors of global solutions arising from high-energy initial data. Based on the existing theories concerning radial solutions, this work precisely describes the asymptotic profiles of unbounded global solutions defined in bounded domains of dimension \(N\ge 3\). Escobedo and  Kavian   in \cite{cr2} conducted a systematic investigation into problem \eqref{e01} and problem \eqref{e02} for $s=0$.
By virtue of the self-similar transformation, problem \eqref{e01} can be reduced to problem \eqref{e02}.
Employing variational methods, including critical point theory and the Mountain Pass Lemma,
they established the existence, multiplicity, symmetry and asymptotic properties of solutions to
the elliptic equation, and further characterized the complete structure of self-similar solutions
 for the heat equation. Thereafter, for $s=0$, problem \eqref{e02} as well as its more general nonlinear
  formulations has been extensively investigated in the literature; see, for instance,
  \cite{dr10,dr17,dr11,dr12,dr13,dr14,dr15,dr16}.

We  denote by $X$ the Hilbert space obtained as the completion of $C_c^{\infty}\left(\mathbb{R}^N\right)$ with respect to the norm
$$\|\nabla v\|_{K,2}:=\left(\int_{\mathbb{R}^N}  |\nabla v|^2 K(y) dy\right)^{\frac{1}{2}}$$
which is induced by the inner product
$$
(v_1, v_2)_K:=\int_{\mathbb{R}^N} (\nabla v_1 \cdot \nabla v_2)K(y) dy.
$$
For each $q \in\left[2,2^{\ast}(s)\right]$,  we  define $L_K^q\left(\mathbb{R}^N\right)$  and  $L_K^q\left(\mathbb{R}^N, \frac{1}{|y|^s}\right)$ as follows:
$$
L_K^q\left(\mathbb{R}^N\right):=\left\{v \text { measurable in } \mathbb{R}^N:\|v\|_{K,q}:=\left(\int_{\mathbb{R}^N}|v|^q  K(y)dy\right)^{1 / q}<\infty\right\},
$$
$$
L_K^q\left(\mathbb{R}^N, \frac{1}{|y|^s}\right):=\left\{v \text { measurable in } \mathbb{R}^N:\left\|v\right\|_{K,q,s}:
=\left(\int_{\mathbb{R}^N} \frac{1}{|y|^s}|v|^q K(y)dy\right)^{1 / q}<\infty\right\} .
$$

The domain $D(L)$ of this operator consists of all $v \in L_K^2\left(\mathbb{R}^N\right)$ such that $L v \in L_K^2\left(\mathbb{R}^N\right)$,
 and one has $D(L)=X$; see Lemma 2.1 in \cite{cr5}. Moreover, $L$ is positive and self-adjoint with a compact inverse.
 In particular, the normalized eigenfunctions of $L$ constitute a complete orthonormal basis in $L_K^2\left(\mathbb{R}^N\right)$.
 The first eigenvalue of $L$ satisfies $\lambda_1=N / 2$, which yields the following Poincar\'{e} inequality
 \begin{equation}\label{po}
  \lambda_1\|v\|_{K, 2}^2 \leqslant\|\nabla v\|_{K, 2}^2, \quad v \in X,
\end{equation}
as stated in Proposition 2.3 of \cite{cr2}.
A function $v\in X$ is a weak solution of
(\ref{e02}), if for every $\phi \in X$ satisfies

$$
 \int_{ \mathbb{R}^N } \nabla v \nabla \phi k(y)dy=\alpha\int_{ \mathbb{R}^N } v\phi K(y)dy+\int_{ \mathbb{R}^N }\frac{|v|^{q-2} v}{|y|^s}\phi K(y)dy.
$$
Weak solutions of (\ref{e02}) are exactly the critical points of the functional $E_K(v) : X \rightarrow R $ defined as
\begin{equation}\label{ef}
 \begin{aligned}
E_{K}(v)&:=\frac{1}{2} \int_{ \mathbb{R}^N }|\nabla v|^2
 K(y)d y- \frac{\alpha}{2}\int_{ \mathbb{R}^N }|v|^2 K(y)d y-\frac{1}{q} \int_{\mathbb{R}^N }\frac{1}{|y|^s}|v|^{q} K(y)dy\\
&=\frac{1}{2} \int_{ \mathbb{R}^N }\left(|\nabla v|^2
-\alpha|v|^2\right) K(y)d y-\frac{1}{q} \int_{\mathbb{R}^N }\frac{1}{|y|^s}|v|^{q} K(y)dy\\
&:=\frac{1}{2}A(v)-\frac{1}{q}B(v).
\end{aligned}
\end{equation}
It is easy to prove that $E_K(v)$ is Fr\'{e}chet differentiable in $X$ and
$$\langle E_K(v),\phi\rangle
= \int_{ \mathbb{R}^N } \nabla v \nabla \phi k(y)dy-\alpha\int_{ \mathbb{R}^N } v\phi K(y)dy-\int_{ \mathbb{R}^N }\frac{|v|^{q-2} v}{|y|^s}\phi K(y)dy.
$$

This paper is organized as follows. In Section 2, we establish the weighted Hardy inequality
and the weighted Sobolev inequality. In Section 3, we investigate the Hardy subcritical singular
case and obtain the existence of infinitely many solutions for this case. In Section 4,
we discuss the Hardy critical singular case and derive the existence and nonexistence results of solutions.
In Section 5, we provide an alternative proof for the existence result established in Section 4.

\section{Main lemmas}
In this section, we mainly prove the K-Hardy Inequality and the K-Sobolev-Hardy Inequality,
 and list several useful lemmas.%We first present the following Hardy inequality with the weighted function $K$.

\begin{lemma}\label{l00}
(See \cite{cr2}) The embedding $X \hookrightarrow L_K^q\left(\mathbb{R}^N\right)$ is continuous for all $q \in\left[2,2^{\ast}(0)\right]$ and it is compact for all $q \in\left[2,2^{\ast}(0)\right)$.
\end{lemma}

\begin{lemma}(See \cite[Lemma 1.21]{cr2})\label{l000}
 (Hardy inequality.) Let $a>-N, N \geqslant 1$, then for any $R \geqslant 0$ and any $u \in C_c^1\left(\mathbb{R}^{N}\right)$
$$
\int_{|x| \geqslant R}|u(x)|^2|x|^a d x \leqslant \frac{4}{(a+N)^2} \int_{x \geqslant R}|x \cdot \nabla u(x)|^2|x|^a d x,
$$
when $R=0$, the inequality is strict unless $u=0$, and the constant
$
\frac{4}{(a+N)^2}
$
is the best possible.
\end{lemma}

\begin{lemma}\label{l01}
 Assume  $u \in H_{K}^{1}\left(\mathbb{R}^N\right)$. Then we have
\begin{description}
  \item[(1)] $u /|x| \in L^2_{K}\left(\mathbb{R}^N\right)$.
  \item[(2)] (K-Hardy inequality)
\begin{equation}\label{wh}
  \int_{\mathbb{R}^N} \frac{|u|^2}{|x|^2} K(x)d x \leqslant C_{N} \int_{\mathbb{R}^N}|\nabla u|^2 K(x)d x
\end{equation}
with $C_{N}=\left(\frac{2}{N-2}\right)^2 $.
  \item[(3)] The constant $C_{N}$ is optimal.

\end{description}
\end{lemma}
\begin{proof}
  By a density argument, it suffices to consider only smooth functions $u \in C_0^{\infty}\left(\mathbb{R}^N\right)$. Then we have
  $$
\begin{aligned}
  |u(x)|^2K(x)&=-\int_1^{\infty} \frac{d}{d \lambda}\left(|u(\lambda x)|^2  K(\lambda x)\right)d \lambda\\
 & =-2 \int_1^{\infty} u(\lambda x)\langle x, \nabla u(\lambda x)\rangle  K(\lambda x) d \lambda-\int_1^{\infty} \left(\frac{\lambda |x|^2}{2}|u(\lambda x)|^2  K(\lambda x)\right)d \lambda.
 \end{aligned}
$$
 By using H\"{o}lder inequality,  we obtain

  $$
\begin{aligned}
\int_{\mathbb{R}^N} \frac{|u(x)|^2}{|x|^2} K(x)d x & =-2 \int_1^{\infty} \int_{\mathbb{R}^N}
\frac{u(\lambda x)}{|x|}\left\langle\frac{x}{|x|}, \nabla u(\lambda x)\right\rangle K(\lambda x)d x d \lambda \\
&\ \ \ \ -\int_1^{\infty}\int_{\mathbb{R}^N} \left(\frac{\lambda }{2}|u(\lambda x)|^2  K(\lambda x)\right)d \lambda.\\
& \leq -2 \int_1^{\infty} \frac{d \lambda}{\lambda^{N-1}} \int_{\mathbb{R}^N} \frac{u(y)}{|y|} \frac{\partial u(y)}{\partial r} K\left(y\right)d y \\
& =-\frac{2}{N-2} \int_{\mathbb{R}^N} \frac{u(y)}{|y|} \frac{\partial u(y)}{\partial r} K\left(y\right) d y \\
& \leqslant \frac{2}{N-2}\left(\int_{\mathbb{R}^N} \frac{|u(y)|^2}{|y|^2}K\left(y\right) d y\right)^{1 / 2}\left(\int_{\mathbb{R}^N}\left|\frac{\partial u(y)}{\partial r}\right|^2 K\left(y\right)d y\right)^{1 / 2}.
\end{aligned}
$$
And then we have
$$
\int_{\mathbb{R}^N} \frac{|u(x)|^2}{|x|^2} K(x) d x\leq
\left(\frac{2}{N-2}\right)^2 \int_{\mathbb{R}^N}|\nabla u(x)|^2  K(x) d x.
$$

Next, we prove that $C_N$ is the optimal constant.
Given $\varepsilon>0$, we take the following radial function
$$
V(r)= \begin{cases}1 & \text { if } r \in[0,1], \\  r^{-\gamma} e^{-\frac{|r|^2}{4}} & \text { if } r>1,\end{cases}
$$
where   $\gamma=\frac{N-2+2 \varepsilon}{2}$. Hence, we have %$B_{N,  \varepsilon}=2 /(N-2+2 \varepsilon)$,
$$
V^2(x)= \begin{cases}1 & \text { if } r \in[0,1], \\ r^{-2\gamma} e^{-\frac{|r|^2}{2}} & \text { if } r>1.\end{cases}
$$
For $r>1$, we obtain
\begin{equation*}%\label{hav}
 \begin{aligned}
V^{\prime}(r)&=-\gamma r^{-\gamma-1} e^{-\frac{r^2}{4}}+r^{-\gamma} e^{-\frac{r^2}{4}} \cdot\left(-\frac{r}{2}\right)=-e^{-\frac{r^2}{4}} r^{-\gamma}\left(\frac{\gamma}{r}+\frac{r}{2}\right) =-V(r)\left(\frac{\gamma}{r}+\frac{r}{2}\right),\\
|\nabla V(x)|^2&=\left(V^{\prime}(r)\right)^2=V^2(x)\left(\frac{\gamma^2}{|x|^2}+\gamma+\frac{|x|^2}{4}\right),\\
\frac{V^2(x)}{|x|^2}&=\frac{|\nabla V(x)|^2}{|x|^2\left(\frac{\gamma^2}{|x|^2}+\gamma+\frac{|x|^2}{4}\right)}
=\frac{|\nabla V(x)|^2}{\left(\gamma^2+\gamma|x|^2+\frac{|x|^4}{4}\right)}\geq \frac{|\nabla V(x)|^2}{\gamma^2}.
\end{aligned}
\end{equation*}
Thus, from the above relation, we get
\begin{equation}\label{hav1}
 \begin{aligned}
\int_{\mathbb{R}^N} \frac{V^2(x)}{|x|^2}K(x) d x & =\int_B \frac{V^2(x)}{|x|^2} K(x)d x+\int_{\mathbb{R}^N\backslash B} \frac{V^2(x)}{|x|^2} K(x)d x \\
& \geq\omega_N \int_0^1 r^{N-3} K(r)d r+ \frac{1}{\gamma^2}\int_{\mathbb{R}^N}|\nabla V(x)|^2K(x) d x.\\
\end{aligned}
\end{equation}
where $\omega_N$ is the measure of the $(N-1)$-dimensional unit sphere. %Letting $\varepsilon \rightarrow 0$, we may therefore conclude this result.

The inequality \eqref{hav1} implies that as long as the constant $1/\gamma^2$ is less than $C_N$,
we can always construct a function $V(x)$ such that the inequality \eqref{wh} fails to hold. Consequently, $C_N$ is shown to be optimal.

\end{proof}

\begin{lemma}\label{l02}
 (K-Sobolev-Hardy Inequality). Assume that  $2 \leq q \leq 2^*(s):=\frac{2N-2s}{N-2}$. Then:
 \begin{description}
   \item[(1)]  There exists a constant $C>0$ such that for any $u \in X$,

$$
\left(\int_{\mathbb{R}^N} \frac{|u|^q}{|x|^s} K(x)d x\right)^2 \leq C\left(\int_{\mathbb{R}^N}|\nabla u|^2K(x)d x\right)^q.
$$
   \item[(2)] The embedding $X \hookrightarrow L_K^q\left(\mathbb{R}^N, \frac{1}{|x|^s}\right)$ is continuous for all $q \in\left[2,2^{\ast}(s)\right]$ and it is compact for all $q \in\left[2,2^{\ast}(s)\right)$.
 \end{description}
\end{lemma}

\begin{proof}
(1) For $s=0$, this result is given by Lemma \ref{l00}. For $s=2$, This is a consequence of Lemma \ref{l01}.
Therefore, we consider the case of $0<s<2$. By Lemma \ref{l00}, Lemma \ref{l01} and H\"{o}lder inequality, we have
\begin{align*}
&\int_{\mathbf{R}^n} \frac{|u|^{2^*(s)}}{|x|^s}  K(x) dx=\int_{\mathbb{R}^n} \frac{|u|^s}{|x|^s} (K(x))^{\frac{s}{2}}  \cdot|u|^{2^*(s)-s} (K(x))^{^{\frac{2-s}{2}}}dx\\
& \leq\left(\int_{\mathbb{R}^n}  \frac{|u|^2}{|x|^2}K(x)dx\right)^{\frac{s}{2}}\left(\int_{\mathbb{R}^n}|u|^{\left(2^*(s)-s\right)\frac{2}{2-s}}K(x)dx\right)^{\frac{2-s}{2}} \\
& =\left(\int_{\mathbb{R}^n}  \frac{|u|^2}{|x|^2}K(x)dx\right)^{\frac{s}{2}}\left(\int_{\mathbb{R}^n}|u|^{\frac{2N}{N-2}}K(x)dx\right)^{\frac{2-s}{2}} \\
& \leq C\left(\int_{\mathbb{R}^n}|\nabla u|^2K(x)\right)^{\frac{s}{2}}\left(\int_{\mathbb{R}^n}|\nabla u|^2K(x)\right)^{\frac{2^*}{2} \cdot \frac{2-s}{2}} \\
& =C\left(\int_{\mathbb{R}^n}|\nabla u|^2K(x)dx\right)^{\frac{N-s}{N-2}}\\
& =C\left(\int_{\mathbb{R}^n}|\nabla u|^2K(x)dx\right)^{\frac{2^{\ast}(s)}{2}}.
\end{align*}
By similar arguments for $2 \leq q < 2^*(s)$, we can also  obtain
$$
\left(\int_{\mathbb{R}^N} \frac{|u|^q}{|x|^s} K(x)d x\right) \leq C\left(\int_{\mathbb{R}^N}|\nabla u|^2K(x)d x\right)^\frac{q}{2}.
$$

(2)  Let $\{u_n\} \subset X$  such that

$$
u_n \rightharpoonup u \quad \text{in}\  X
$$
and, for instance, $\left\|\nabla u_n\right\|_{K,2} \leqslant 1$. Since  $X \subset H^1\left(\mathbb{R}^N\right)$,
 one can suppose that (by Rellich's compactness theorem):
$$
u_n \rightarrow u \quad \text { in } L_{\text {loc }}^2\left(\mathbb{R}^N,\frac{1}{|x|^s}\right) .
$$
Then , we  have
$$
\int_{\mathbb{R}^N}\left|u_n-u\right|^2 \frac{1}{|x|^s}K(x) d x=\int_{|x| \leqslant R}\left|u_n-u\right|^2 \frac{1}{|x|^s}K(x)dx +\int_{|x|>R}\left|u_n-u\right|^2 \frac{1}{|x|^s}K(x)dx.
$$
Given an $\varepsilon>0$, we  choose $R$ large enough such that
\begin{equation}\label{s3}
 \left(\frac{N+s}{2}|x|^s-s(N+s-2)|x|^{s-2}+\left(\frac{1}{4}-\frac{s}{2|x|^2}\right)^2 |x|^{2+s}\right)\geqslant 2 / \varepsilon , \ \ |x| \geqslant R.
\end{equation}
 Let $v:=|x|^{-\frac{s}{2}}e^{\frac{|x|^2}{8}}u$. Then one sees easily that
$$
\begin{aligned}
\nabla v= |x|^{-\frac{s}{2}} e^{\frac{|x|^2}{8}} \nabla u +e^{\frac{|x|^2}{8}}|x|^{-\frac{s}{2}}\left(\frac{1}{4}-\frac{s}{2|x|^2}\right) x u,
\end{aligned}
$$
that is,
\begin{equation}\label{s7}
 \begin{aligned}
  e^{\frac{|x|^2}{8}} \nabla u =|x|^{\frac{s}{2}} \nabla v-e^{\frac{|x|^2}{8}}\left(\frac{1}{4}-\frac{s}{2|x|^2}\right) x u
 =|x|^{\frac{s}{2}}\nabla v-|x|^{\frac{s}{2}}v\left(\frac{1}{4}-\frac{s}{2|x|^2}\right) x.
\end{aligned}
\end{equation}
Integrating both sides of the square of  (\ref{s7}) yields

\begin{equation}\label{s5}
\begin{aligned}
 \int_{\mathbb{R}^N}  e^{\frac{|x|^2}{4}} |\nabla u|^2dx & =\int_{\mathbb{R}^N} |\nabla v|^2|x|^sdx
 +\int_{\mathbb{R}^N} v^2\left(\frac{1}{4}-\frac{s}{2|x|^2}\right)^2 |x|^2|x|^{s}dx\\
 &\ \ \ -2\int_{\mathbb{R}^N} \left(\frac{1}{4}-\frac{s}{2|x|^2}\right)|x|^{s} x\nabla v^2dx.
\end{aligned}
\end{equation}
And integrating  the RHS cross term of \eqref{s5}  by parts, we obtain
\begin{equation}\label{s4}
  \begin{aligned}
  -2\int_{\mathbb{R}^N}\left(\frac{1}{4}-\frac{s}{2|x|^2}\right)|x|^s x\nabla v^2dx
  &=2\int_{\mathbb{R}^N} \operatorname{div}\left(\left(\frac{1}{4}-\frac{s}{2|x|^2}\right)|x|^{s} x\right)v^2dx\\
  &=\int_{\mathbb{R}^N} \left(\frac{N+s}{2}|x|^s-s(N+s-2)|x|^{s-2}\right)v^2dx.
  \end{aligned}
\end{equation}
Hence,by (\ref{s5}) and (\ref{s4}) we have
\begin{equation}\label{s6}
 \begin{aligned}
 \int_{\mathbb{R}^N}  e^{\frac{|x|^2}{4}} |\nabla u|^2dx&\geq\int_{\mathbb{R}^N}
 \left(\frac{N+s}{2}|x|^s-s(N+s-2)|x|^{s-2}+\left(\frac{1}{4}-\frac{s}{2|x|^2}\right)^2 |x|^{2+s}\right)v^2dx.
  \end{aligned}
\end{equation}
Then using this together with \eqref{s3} we get
$$
\int_{|x| \geqslant R}\left|u_n-u\right|^2 |x|^{-s}K(x) d x
\leqslant \varepsilon C \int_{|x| \geqslant R}\left|\nabla\left(u_n-u\right)\right|^2 |x|^{-s}K(x)d x \leqslant \varepsilon C
$$
uniformly on $N \geqslant 3$.

On the other hand,  According to  Lemma 3.2 of \cite{cr1}, for fixed $R$, we may choose sufficiently large $n$ such that
$$
 \int_{|x|<R}\left|u_n-u\right|^2 K(x) \frac{1}{|x|^s}dx\leqslant \varepsilon.
$$
Hence $u_n \rightarrow u$ in $L_K^2\left(\mathbb{R}^N, \frac{1}{|x|^s}\right)$.

Now assume that  $\theta$ satisfies $$
\frac{1}{q}=\frac{\theta}{2}+\frac{1-\theta}{2^{\ast}(s)}, \text{\ that\  is,\  }   1=\frac{\theta q}{2}+\frac{(1-\theta)q}{2^{\ast}(s)}.$$
Then  the H\"{o}lder inequality implies that
\begin{equation}\label{s8}
  \begin{aligned}
\int_{\mathbb{R}^N}\left|u_n-u\right|^q \frac{K(x)}{|x|^s} d x&=\int_{\mathbb{R}^N}\left|u_n-u\right|^{\theta q} \left[\frac{K(x)}{|x|^s}\right]^{\frac{\theta q}{2}} \left|u_n-u\right|^{(1-\theta)q} \left[\frac{K(x)}{|x|^s}\right]^{\frac{(1-\theta)q}{2^{\ast(s)}}}d x\\
&\leq\left(\int_{\mathbb{R}^N}\left|u_n-u\right|^2 \frac{K(x)}{|x|^s} d x\right)^\frac{\theta q}{2}
\left(\int_{\mathbb{R}^N}\left|u_n-u\right|^{2^{\ast}(s)} \frac{K(x)}{|x|^s} d x\right)^{\frac{(1-\theta)q}{2^{\ast(s)}}} .
\end{aligned}
\end{equation}
Combining Lemma \ref{l02} (1) with the boundedness of  $u_n$ and $u$ in $X$, we easily obtain $u_n \rightarrow u$ in $L_K^q\left(\mathbb{R}^N, \frac{1}{|x|^s}\right)$.

\end{proof}
\begin{lemma}(\cite{cr3})\label{bl}
Suppose $v_n \rightarrow v$ a.e. in $\mathbb{R}^N$ and $\left\|v_n\right\|_{K,q} \leq C<\infty$ for all $n$ and for some $1<q<\infty$. Then
$$
\|v_n-v\|_{K,q}^{q}=\|v_n\|_{K,q}^{q}-\|v\|_{K,q}^{q}+o(1).
$$
\end{lemma}

\begin{lemma}(\cite{cr1})\label{b2}
Suppose $v_n \rightarrow v$ a.e. in $\mathbb{R}^N$ and $\left\|v_n\right\|_{K,q,s} \leq C<\infty$ for all $n$ and for some $1<q<\infty$. Then
$$
\|v_n-v\|_{K,q,s}^{q}=\|v_n\|_{K,q,s}^{q}-\|v\|_{K,q,s}^{q}+o(1).
$$
\end{lemma}

\section{Hardy subcritical singular case}
In this section, we deal with the Hardy subcritical singular  case and employ the fountain theorem to establish the existence of infinitely many solutions.
Our main result is as follow.
\begin{theorem}\label{t31}
 Suppose $\alpha\in \mathbb{R}$ and  $2 < q<2^*(s)$.
Then, problem \eqref{e02} has solutions $\{\pm v_{k}\}_{k=1}^{\infty}$ such that $E_{K}(\pm v_{k})\to+\infty\ \mathrm{as\ }k\ \to\infty$.
\end{theorem}

This result shows that the heat equation \eqref{e01} admits infinitely many self-similar solutions, and

$$\int_{\mathbb{R}^N} |u|^q dx= t^{-\alpha} \int_{\mathbb{R}^N} |v|^qdy\rightarrow 0,\ \ \text{as}\ \  t\rightarrow+\infty. $$

As $X$ is a separable and reflexive Banach space, there exist $\{e_n\}_{n=1}^{\infty}\subset X$ and $\{f_n\}_{n=1}^{\infty}\subset X^*$ such that
$$f_n(e_m)=\delta_{n,m}=
\begin{cases}
1 & \text{if } n=m, \\
0 & \text{if } n\ne m,
\end{cases}$$
$$X=\overline{\text{span}}\{e_n:n=1,2,\dots\},\quad X^*=\overline{\text{span}}^{W^*}\{f_n:n=1,2,\dots\}.$$
For $k=1,2,\dots$, denote
\begin{equation}\label{e331}
  X_k=\text{span}\{e_k\},\ Y_k=\bigoplus_{j=1}^k X_j,\ Z_k=\overline{\bigoplus_{j=k}^\infty X_j}.
\end{equation}
\begin{lemma}\label{ll31}
 (Fountain theorem, see \cite{cr4}). Assume that
 \begin{description}
   \item[(A1)] $X$ is a Banach space, $I\in C^1(X,\mathbb{R})$ is an even functional,
   the subspaces $X_k$, $Y_k$ and $Z_k$ are defined by (\ref{e331}).

   If for each $k=1,2,\dots$, there exists $\rho_k>r_k>0$ such that
   \item[(A2)]  $\inf_{u\in Z_k,\ \|u\|=r_k}I(u)\to +\infty$ as $k\to\infty$,
   \item[(A3)] $\max_{u\in Y_k,\ \|u\|=\rho_k}I(u)\leqslant 0$,
   \item[(A4)] $I$ satisfies $(\mathrm{PS})_c$ condition for every $c>0$.

   Then $I$ has a sequence of critical values tending to $+\infty$.
 \end{description}
\end{lemma}

\begin{lemma}\label{ll32}
 Assume that the functional $E_K(v)$ is defined as in (\ref{ef}),  $N\geq 3$. Then for every $\beta\in \mathbb{R}$,
  every sequence $\left\{v_n\right\}$ in $X$ such that, as $n \rightarrow \infty$,

\begin{equation}\label{o01}
E_K\left(v_n\right) \rightarrow \beta\in \mathbb{R}, \quad \langle E_{K}^{\prime}(v_n), v_n\rangle \rightarrow 0,
\end{equation}
is relatively compact in $X$.
\end{lemma}

\begin{proof}
It is easy to check that $\left\{v_n\right\}$ is bounded in $X$. In fact,
$$
\begin{aligned}
&\beta+o(1)\left(1+\left\|\nabla v_n\right\|_{K, 2}\right) \geq E_K\left(v_n\right)-\frac{1}{2}\left\langle E_K^{\prime}\left(v_n\right), v_n\right\rangle \\
& \quad=\left(\frac{1}{2}-\frac{1}{q}\right) \int_{\mathbb{R}^N} \frac{1}{|y|^s}\left|v_n\right|^{q} K(y) d y.
\end{aligned}
$$
So
$$
\begin{aligned}
\left(1-\frac{N-2}{4 \lambda_1}\right) \int_{\mathbb{R}^N}\left|\nabla v_n\right|^2 K(y) d y & \leq \int_{\mathbb{R}^N}\left|\nabla v_n\right|^2 K(y) d y-\frac{N-2}{4} \int_{\mathbb{R}^N}\left|v_n\right|^2 K(y) d y \\
& =2 E_K\left(v_n\right)+\frac{2}{q} \int_{\mathbb{R}^N} \frac{1}{|y|^s}|v|^{q} K(y) d y \\
& \leq C_2+o(1)\left\|\nabla v_n\right\|_{K, 2},
\end{aligned}
$$
where $o(1) \rightarrow 0$ as $n \rightarrow \infty$. Thus it follows that $\left\{v_n\right\}$ is bounded in $X$. So we can extract a subsequence, still denoted by $\left\{v_n\right\}$, such that, as $n \rightarrow \infty$,

\begin{equation}
\begin{array}{ll}
v_n \rightharpoonup v & \text { weakly in } X, \\
v_n \rightarrow v & \text { in } L_K^2\left(\mathbb{R}^N\right), \\
v_n \rightarrow v & \text { in } L_K^q\left(\mathbb{R}^N, \frac{1}{|y|^s}\right), \\
v_n \rightarrow v & \text { a.e. on } \mathbb{R}^N .
\end{array}
\end{equation}
Observe that
\begin{align*}
\left\|\nabla v_n-\nabla v\right\|_{K,2}^2&=\left\langle E_K^{\prime}\left(v_n\right)-E_K^{\prime}(v), v_n-v\right\rangle_K\\
&+\int_{\mathbb{R}^N}\left[ |v_n|^{q-2}v_n-|v|^{q-2} v\right](v_n-v) \frac{1}{\left|y\right|^s}K(y)d y+\alpha\int_{\mathbb{R}^N} (v_n-v)^2 \frac{1}{\left|y\right|^s} K(y)d y.
\end{align*}
It is clear that
$$
\left\langle E_K^{\prime}\left(v_n\right)-E_K^{\prime}(v), v_n-v\right\rangle_K\rightarrow 0,
$$
and
$$\int_{\mathbb{R}^N} (v_n-v)^2 \frac{1}{\left|y\right|^s} K(y)d y \rightarrow 0.$$
It follows from the H\"{o}lder inequality  and continuity of the Nemytskii operator (see \cite[Theorem A.2]{cr4}) that
\begin{align*}
&\int_{\mathbb{R}^N}\left[\left|v_n\right|^{q-2} v_n-|v|^{q-2} v\right]\left(v_n-v\right) \frac{1}{|y|^s} K(y) d y\\
&\leq \left(\int_{\mathbb{R}^N}\left[\left|v_n\right|^{q-2} v_n-|v|^{q-2} v\right]^{\frac{q}{q-1}} \frac{1}{|y|^s} K(y) d y\right)^{\frac{q-1}{q}}
\left[\int_{\mathbb{R}^N}\left(v_n-v\right)^{q} \frac{1}{|y|^s} K(y) d y\right]^{\frac{1}{q}}\rightarrow 0.
\end{align*}
Thus we have proved that $\left\|\nabla v_n-\nabla v\right\|_K \rightarrow 0, n \rightarrow \infty$.

\end{proof}

\begin{proof}[Proof of Theorem \ref{t31}]
 Since on the finite-dimensional space $Y_k$ all norms are equivalent, relation (A3) is satisfied for every $\rho_k>0$ large enough.
Let us define
$$
\beta_k:=\sup _{\substack{v \in Z_k \\\|\nabla v\|_K=1}}\left[\int_{\mathbb{R}^N} \frac{1}{|y|^s}\left|v\right|^q K(y) d y\right]^{\frac{1}{q}},
$$
and
$$\gamma_k:=\sup _{\substack{v \in Z_k \\\|\nabla v\|_K=1}}\left[\int_{\mathbb{R}^N} \left|v\right|^2 K(y) d y\right]^{\frac{1}{2}}.
$$
We have $$\beta_k\rightarrow 0, \gamma_k\rightarrow 0.$$
In fact, we know that  $0<\beta_{k+1} \leq \beta_k$, so that $\beta_k \rightarrow \beta \geq 0, k \rightarrow \infty$.
For every $k \geq 0$, there exists $ v_k \in Z_k$ such that $\left\|\nabla v_k\right\|_{K}=1$ and $\left\|v_k\frac{1}{|y|^s}\right\|_{K,q}>\beta_k / 2$.
By definition of $Z_k, v_k \rightharpoonup 0$ in $X$.
Lemma \ref{l02} implies that $v_k \rightarrow 0$ in $L_K^q\left(\mathbb{R}^N, \frac{1}{|y|^s}\right)$. Thus we have proved that $\beta=0$.
Similarly, $\gamma_k\rightarrow 0.$

Then, on $Z_k$, for sufficiently large $k$, we have
\begin{align*}
E_{K}(v)&\geq \frac{1}{2} \int_{ \mathbb{R}^N }|\nabla v|^2 K(y)d y-\frac{\alpha}{2} \int_{ \mathbb{R}^N }| v|^2 K(y)d y-\frac{1}{q} \int_{\mathbb{R}^N }\frac{1}{|y|^s}|v|^{q} K(y)dy\\
&\geq \frac{1}{2} \|\nabla v\|^2_{K,2}-\frac{\alpha}{2}c_1\gamma_k^2 \|\nabla  v\|^2_{K,2}- c_2\beta_k^q \|\nabla v\|^q_{K,2}\\
&\geq \frac{1}{4} \|\nabla v\|^2_{K,2}- c_2\beta_k^q \|\nabla v\|^q_{K,2}.
\end{align*}
Choosing $r_k:=\left(2c_2 q \beta_k^q\right)^{1 /(2-q)}$, we obtain, if $v \in Z_k$ and $\|\nabla v\|_{K,2}=r_k$,
$$
E_{K}(v) \geq\frac{q-2}{4 q}\left(\frac{1}{2 q c_2 \beta_k^q}\right)^{\frac{2}{q-2}}.
$$
Since $\beta_k \rightarrow 0$, relation (A2) is proved.
\end{proof}

\section{Hardy critical singular case}
In this section, we prove the nonexistence and existence of solutions to problem (\ref{e02}) in the hardy critical singular case case.
We define
$$
Q_{\alpha}(v)=\frac{\int_{\mathbb{R}^N} |\nabla v|^2dy-\alpha \int_{\mathbb{R}^N}  v^2dy}
{\left(\int_{\mathbb{R}^N} |v|^{2^*(s)}\frac{1}{|y|^s}dy\right)^{2 / 2^*(s)}},
$$
$$
Q_{K,\alpha}(v)=\frac{\int_{\mathbb{R}^N} K(y)|\nabla v|^2dy-\alpha \int_{\mathbb{R}^N} K(y) v^2dy}
{\left(\int_{\mathbb{R}^N} K(y)|v|^{2^*(s)}\frac{1}{|y|^s}dy\right)^{2 / 2^*(s)}},
$$
and
\begin{equation}\label{ska}
S_{\alpha}=\inf _{v \in X \backslash\{0\}} Q_{\alpha}(v), \quad S_{K,\alpha}=\inf _{v \in X \backslash\{0\}} Q_{K,\alpha}(v).
\end{equation}
Our main results are as follows:
\begin{theorem}\label{t41}
Let $\alpha \leqslant N / 4 (N \geqslant 3)$. Suppose that $v\in X$ satisfies problem \eqref{e02}, then $v \equiv 0$.
\end{theorem}

\begin{theorem}\label{t42}
\begin{description}
  \item[(1)] Suppose $N \geqslant 4$. Then for any $\left.\alpha \in\right(N / 4, N / 2)$ there is at least one
  nontrivial solution of \eqref{e02}, and if $\alpha\not\in ( N / 4, N / 2)$
   there is no positive solution of \eqref{e02};
  \item[(2)] If $N=3$, for any $\alpha  \in( 1,3 / 2)$ there is at least one nontrivial solution of \eqref{e02}.
  For $\alpha  \leqslant 3 / 4$ or $\alpha \geqslant 3 / 2$ equation \eqref{e02} has no positive solution.
\end{description}

\end{theorem}

\begin{proof}[Proof of Theorem \ref{t41}]
 As is customary, the proof is based on an appropriate Pohozaev identity.

 Firstly, we  choose $\varphi \in C_c^1\left(\mathbb{R}^{N}\right)$, such that $0\leq \varphi(y) \leq 1$ and
\begin{equation}
\varphi(y)=\begin{cases}
1 & \text { if } |y| \leqslant 1, \\
0 & \text { if } |y| \geqslant 2.
\end{cases}
\end{equation}
Define for $n \geqslant 1, \varphi_n(y):=\varphi(y / n)$.
Multiplying \eqref{e02} by $v \varphi_n$ and integrating by parts we get
$$
\int_{\mathbb{R}^N}|\nabla v|^2 \varphi_ndy +\int_{\mathbb{R}^N} \nabla v \cdot \nabla \varphi_n \cdot vdy
-\frac{1}{4} \int_{\mathbb{R}^N} \nabla v^2 \cdot y \varphi_n dy
=\int_{\mathbb{R}^N}\left(\frac{|v|^{2^{\ast}(s)}}{|y|^s}+\alpha|v|^2\right)  \varphi_ndy.
$$
Letting $n \rightarrow \infty$,
noting that the fact that $$y \cdot \nabla \varphi_n=(y / n) \cdot \nabla \varphi(y / n),|y||\nabla \varphi(y)| \leqslant C$$
and $$(y / n) \cdot \nabla \varphi(y / n) \rightarrow 0$$
and by  Gauss divergence theorem,  it is easy to see that

$$
\begin{aligned}
&\int_{\mathbb{R}^N}|\nabla v|^2 \varphi_ndy\rightarrow \int_{\mathbb{R}^N}|\nabla v|^2 dy,\\
&\int_{\mathbb{R}^N}\left(|v|^{2^*(s)}+\alpha|v|^2\right) \cdot \varphi_n d y\rightarrow \int_{\mathbb{R}^N}\left(\frac{|v|^{2^*(s)}}{|y|^s}+\alpha|v|^2\right) d y,\\
&\int_{\mathbb{R}^N} \nabla v \cdot \nabla \varphi_n \cdot v d y\rightarrow 0,\\
&-\frac{1}{4} \int_{\mathbb{R}^N} \nabla v^2 \cdot y \varphi_n d y
=\frac{N}{4} \int_{\mathbb{R}^N} v^2 \varphi_ndy+\int_{\mathbb{R}^N} v^2 y \cdot \nabla \varphi_ndy \rightarrow \frac{N}{4} \int_{\mathbb{R}^N} v^2 dy.
\end{aligned}
$$
Now  we obtain
\begin{equation}\label{sk11}
 \int_{\mathbb{R}^N}|\nabla v|^2  d y+\frac{N}{4} \int_{\mathbb{R}^N} v^2 d y=\int_{\mathbb{R}^N}\left(\frac{|v|^{2^*(s)}}{|y|^s}+\alpha|v|^2\right) d y.
\end{equation}
 Next multiplying \eqref{e02} by $y_k \partial_k v \cdot \varphi_n$ (for $1 \leq k \leq N$ fixed) and integrating by parts one gets
\begin{equation}\label{sk8}
\begin{aligned}
& \int_{\mathbb{R}^N}\left(\nabla v \cdot \nabla \partial_k v\right) y_k \varphi_ndy+\int_{\mathbb{R}^N}\left|\partial_k v\right|^2 \varphi_n\\
&\quad+\int_{\mathbb{R}^N} \nabla v \cdot \nabla \varphi_n\left(\partial_k v\right) y_kdy
-\frac{1}{2} \int_{\mathbb{R}^N}(y \cdot \nabla v)\left(\partial_k v \cdot y_k\right) \varphi_ndy\\
&=\int_{\mathbb{R}^N} \left( \partial_k\left(\frac{|v|^{2^{\ast}(s)}}{2^{\ast}(s)|y|^s}+\frac{\alpha}{2}|v|^2\right)y_k \varphi_n+\frac{s|v|^{2^{\ast}(s)}}{2^{\ast}(s)}|y|^{-s-2}y_k^2\varphi_n\right)dy.
\end{aligned}
\end{equation}
Noting that
\begin{equation}\label{sk9}
\begin{aligned}
\int_{\mathbb{R}^N}\left(\nabla v \cdot \nabla \partial_k v\right) y_k \varphi_ndy
&=\frac{1}{2} \int_{\mathbb{R}^N} \partial_k|\nabla v|^2 \cdot y_k \varphi_ndy\\
&=-\frac{1}{2} \int_{\mathbb{R}^N}|\nabla v|^2 \varphi_ndy-\frac{1}{2} \int_{\mathbb{R}^N}|\nabla v|^2 y_k \cdot \partial_k \varphi_ndy.
\end{aligned}
\end{equation}
and the first term on the RHS of (\ref{sk8})
\begin{equation}\label{sk10}
\begin{aligned}
\int_{\mathbb{R}^N}\partial_k\left(\frac{|v|^{2^*}(s)}{2^*(s)|y|^s}+\frac{\alpha}{2}|v|^2\right) y_k \varphi_ndy
&=-\int_{\mathbb{R}^N}\left(\frac{|v|^{2^*}(s)}{2^*(s)|y|^s}+\frac{\alpha}{2}|v|^2\right)  \varphi_ndy\\
&-\int_{\mathbb{R}^N}\left(\frac{|v|^{2^*(s)}}{2^*(s)|y|^s}+\frac{\alpha}{2}|v|^2\right) y_k \partial_k\varphi_ndy.
\end{aligned}
\end{equation}
Letting $n\rightarrow \infty$ in \eqref{sk8},  using \eqref{sk9} and \eqref{sk10},  we  obtain

\begin{equation*}
\begin{aligned}
&-\frac{1}{2} \int_{\mathbb{R}^N}|\nabla v|^2  d y+\int_{\mathbb{R}^N}\left|\partial_k v\right|^2 dy
-\frac{1}{2} \int_{\mathbb{R}^N}(y \cdot \nabla v)\left(\partial_k v \cdot y_k\right)  d y\\
&=-\int_{\mathbb{R}^N}\left(\frac{|v|^{2^*}(s)}{2^*(s)|y|^s}+\frac{\alpha}{2}|v|^2\right)  d y
+\int_{\mathbb{R}^N}\frac{s|v|^{2^*(s)}}{2^*(s)}|y|^{-s-2} y_k^2  d y .
\end{aligned}
\end{equation*}
Summing the above identity from 1 to $N$, we obtain

\begin{equation}\label{sk12}
\begin{aligned}
& \frac{N-2}{2} \int_{\mathbb{R}^N}|\nabla v|^2 d y
+\frac{1}{2} \int_{\mathbb{R}^N}(y \cdot \nabla v)^2 d y \\
& =(N-s)\int_{\mathbb{R}^N}\left(\frac{|v|^2}{2^*(s)|y|^s}\right) d y
+N\int_{\mathbb{R}^N}\frac{\alpha}{2}|v|^2 d y.
\end{aligned}
\end{equation}
Combining (\ref{sk11}), (\ref{sk12}), we infer
\begin{equation}\label{sk13}
\frac{1}{2} \int_{\mathbb{R}^N}(y \cdot \nabla v)^2 d y
=\left(\frac{N(N-2)}{8}+\alpha\right)\int_{\mathbb{R}^N} |v|^2 d y.
\end{equation}
Then  Lemma \ref{l000} and \eqref{sk13} implies that
$$
\frac{N^2}{8} \int_{\mathbb{R}^N}|v|^2 d y < \left(\frac{N(N-2)}{8}+\alpha\right) \int_{\mathbb{R}^N}|v|^2 d y .
$$
So, $\alpha>\frac{N}{4}$.
\end{proof}

\begin{lemma}\label{l41}
If $0<S_{K,\alpha}<S_{K,0}$, the infimum in (\ref{ska}) is achieved (and this yields obviously a solution to (\ref{e02}).
\end{lemma}

\begin{proof}
 Let $\left\{v_n\right\}$  a sequence in $X$ such that
\begin{equation}\label{sk1}
\begin{aligned}
\int_{\mathbb{R}^N} K(y)|v_n|^{2^*(s)} \frac{1}{|y|^s} d y & =1,  \\
\int_{\mathbb{R}^N}\left|\nabla v_n\right|^2 K(y)dy-\alpha \int_{\mathbb{R}^N}\left|v_n\right|^2 K(y)dy & =S_{K,\alpha}+o(1).
\end{aligned}
\end{equation}
From Lemma \ref{l00}, one can suppose that
\begin{equation}\label{sk2}
  \begin{array}{ll}
v_n \rightharpoonup  v & \text { weakly in } X, \\
v_n \rightarrow v & \text { in } L_K^2(\mathbb{R}^N), \\
v_n \rightarrow v & \text { a.e. on } \mathbb{R}^N,
\end{array}
\end{equation}
and
\begin{equation}\label{sk5}
  \int_{\mathbb{R}^N} K(y)\left|v\right|^{2^*(s)} \frac{1}{|y|^s} d y\leq 1.
\end{equation}
Now set $w_n:=v_n-v,$ we have
\begin{equation}\label{sk6}
\begin{array}{ll}
w_n \rightharpoonup 0 & \text { weakly in } X, \\
w_n \rightarrow 0 & \text { a.e. on } \mathbb{R}^N .
\end{array}
\end{equation}
By the definition of $S_{K,0}$, we obtain
$$
\int_{\mathbb{R}^N}\left|\nabla v_n\right|^2 K(y)dy \geqslant S_{K,0},
$$
and by (\ref{sk1}) and the fact that $v_n \rightarrow v$ in $L_K^2(\mathbb{R}^N)$
$$
\alpha \int_{\mathbb{R}^N}|v|^2 K(y)dy  \geqslant S_{K,0}-S_{K,\alpha}>0
$$
and therefore $v \neq 0$.

Using  (\ref{sk1}) and (\ref{sk2}),  one has

\begin{equation}\label{sk3}
\int_{\mathbb{R}^N} |\nabla v|^2 K(y)dy+\int_{\mathbb{R}^N}\left|\nabla w_n\right|^2 K(y)dy-\alpha \int_{\mathbb{R}^N}|v|^2 K(y)dy =S_{K,\alpha}+o(1).
\end{equation}
By Brezis-Lieb lemma (see Lemma \ref{b2})
\begin{equation}\label{sk4}
  1=\int_{\mathbb{R}^N}\left|v+w_n\right|^{2^{\ast}(s)} K(y)\frac{1}{|y|^s} dy=\int|v|^{2^{\ast}(s)} K(y)\frac{1}{|y|^s} dy+\int_{\mathbb{R}^N}\left|w_n\right|^{2^{\ast}(s)} K(y)\frac{1}{|y|^s} dy+o(1).
\end{equation}
Since $2 / 2^{\ast}(s)<1$, by \eqref{sk5}, \eqref{sk6}, \eqref{sk4} and the definition of $S_{K,0}$ we deduce that
$$
\begin{aligned}
& 1 \leqslant\left(\int_{\mathbb{R}^N}|v|^{2^{\ast}(s)} K(y)\frac{1}{|y|^s} dy\right)^{2 / 2^{\ast}(s)}+\left(\int_{\mathbb{R}^N}\left|w_n\right|^{2^{\ast}(s)} K(y)\frac{1}{|y|^s} dy\right)^{2 / 2^{\ast}(s)}+o(1), \\
& 1 \leqslant\left(\int_{\mathbb{R}^N}|v|^{2^{\ast}(s)} K(y)\frac{1}{|y|^s} dy\right)^{2 / 2^{\ast}(s)}+\frac{1}{S_{K,0}} \int_{\mathbb{R}^N}\left|\nabla w_n\right|^2 K(y) dy+o(1).
\end{aligned}
$$
Then  $S_{K,\alpha}>0$ implies that
\begin{equation}\label{skkk}
S_{K,\alpha} \leqslant S_{K,\alpha}\left(\int_{\mathbb{R}^N}|v|^{2^*(s)} K(y) \frac{1}{|y|^s} d y\right)^{2 / 2^*(s)}+\frac{S_{K,\alpha}}{S_{K, 0}} \int_{\mathbb{R}^N}\left|\nabla w_n\right|^2 K(y)  d y+o(1).
\end{equation}
Now from \eqref{sk3} and \eqref{skkk} we obtain
\begin{equation*}
\begin{aligned}
&\int_{\mathbb{R}^N}|\nabla v|^2 K(y) d y-\alpha \int_{\mathbb{R}^N}|v|^2 K(y) d y+\left(1-\frac{S_{K, \alpha}}{S_{K, 0}} \right)\int_{\mathbb{R}^N}\left|\nabla w_n\right|^2 K(y) d y\\
&\quad \leq S_{K, \alpha}\left(\int_{\mathbb{R}^N}|v|^{2^*(s)} K(y) \frac{1}{|y|^s} d y\right)^{2 / 2^*(s)}+o(1).
\end{aligned}
\end{equation*}
By the assumption $0<S_{K,\alpha}<S_{K,0}$, we conclude that
$$
Q_{K,\alpha }(v) \leqslant S_{K,\alpha}.
$$
Actually, we can obtain that
$$
\int_{\mathbb{R}^N}\left|\nabla w_n\right|^2 K(y) d y\rightarrow 0.
$$

\end{proof}

We next prove that there exists a range of the parameter $\alpha$-depending on the dimension $N$ such that
$$
S_{K, \alpha} \in\left(0, S_{K,0}\right).
$$
To this end, we have following lemmas:
\begin{lemma}(\cite[Lemma  4.16]{cr2})\label{l42}
 Let $N \geqslant 1$ and
$$
Y:=\left\{v \in H^1\left(\mathbb{R}^N\right) ;|y| v \in L^2\left(\mathbb{R}^N\right)\right\} .
$$
Then $v \in X$ if and only if  $K^{1 / 2} v \in Y$.
\end{lemma}

\begin{lemma}(\cite[Theorem 3.1]{cr1})\label{l43}
 Let $$S_0:=\inf _{v \in H^{1}\left(\mathbb{R}^N\right)\setminus\{0\}}
 \frac{\int_{\mathbb{R}^N}|\nabla v|^2 d y d y}{\left(\int_{\mathbb{R}^N}|v|^{2^*(s)} \frac{1}{|y|^s} d y\right)^{2 / 2^*(s)}}.$$
 Then $S_0$ is attained by the functions

$$
U_{\varepsilon}(y)=\left(\varepsilon\cdot(N-s)\left(N-2\right)\right)^{\frac{N-2}{2(2-s)}}\left(\varepsilon+|y|^{2-s}\right)^{\frac{2-N}{2-s}}
$$
for some $\varepsilon>0$. Moreover the functions $U_\varepsilon$ are the only positive radial solutions of
$$
-\Delta v=\frac{v^{2^*(s)-1}}{|y|^s} \ \text{in}\ \mathbb{R}^N.
$$
 Hence,
$$
S_0\left(\int_{\mathbb{R}^n} \frac{\left|U_\varepsilon\right|^{2^*(s)}}{|y|^s}dx\right)^{\frac{2}{2^*(s)}}=\left\|\nabla U_\varepsilon\right\|_2^2
=\int_{\mathbb{R}^n} \frac{\left|U_\varepsilon\right|^{2^*(s)}}{|y|^s}dx=S_0^{\frac{N-s}{2-s}}.
$$
\end{lemma}

\begin{lemma}\label{l44}
 Let $N \geqslant 3$. Then $S_{K,0}\geq S_0$.
\end{lemma}
\begin{proof}
It is known that
\begin{equation*}
 \begin{aligned}
S_{K,0}=\inf _{X\setminus\{0\}} \frac{\int_{\mathbb{R}^N} K(y)|\nabla v|^2dy}
{\left(\int_{\mathbb{R}^N} K(y)|v|^{2^*(s)}\frac{1}{|y|^s}dy\right)^{2 / 2^*(s)}}.
\end{aligned}
\end{equation*}
By  Gauss divergence theorem
\begin{equation}\label{sk7}
-\frac{1}{2}\int_{\mathbb{R}^N} y \cdot(v \nabla v) d y=-\frac{1}{4}\int_{\mathbb{R}^N} y \cdot \nabla v^2 d y=\frac{N}{4} \int_{\mathbb{R}^N} v^2 d y.
\end{equation}
Set  $v:=K^\frac{-1}{2} w$.  Using  Lemma \ref{l42}, \eqref{sk7} and the fact $K>1$, we have
\begin{equation*}
 \begin{aligned}
S_{K,0}&=\inf _{X\setminus\{0\}} \frac{\int_{\mathbb{R}^N} K(y)|\nabla v|^2dy}
{\left(\int_{\mathbb{R}^N} K(y)|v|^{2^*(s)}\frac{1}{|y|^s}dy\right)^{2 / 2^*(s)}}\\
&\geq \inf _{Y\setminus\{0\}} \frac{\int_{\mathbb{R}^N} |\nabla w-\frac{y}{4}w|^2dy}
{\left(\int_{\mathbb{R}^N} \left[K(y)\right]^{1-\frac{2^{\ast}(s)}{2}}|w|^{2^*(s)}\frac{1}{|y|^s}dy\right)^{2 / 2^*(s)}}\\
&\geq \inf _{Y\setminus\{0\}} \frac{\int_{\mathbb{R}^N} \left(|\nabla w|^2-\frac{1}{2}yw\nabla w+\frac{y^2}{16}w^2\right)dy}
{\left(\int_{\mathbb{R}^N} \left[K(y)\right]^{1-\frac{2^{\ast}(s)}{2}}|w|^{2^*(s)}\frac{1}{|y|^s}dy\right)^{2 / 2^*(s)}}\\
&\geq \inf _{Y\setminus\{0\}} \frac{\int_{\mathbb{R}^N} \left(|\nabla w|^2+\left(\frac{y^2}{16}+\frac{N}{4}\right)w^2\right)dy}
{\left(\int_{\mathbb{R}^N} \left[K(y)\right]^{1-\frac{2^{\ast}(s)}{2}}|w|^{2^*(s)}\frac{1}{|y|^s}dy\right)^{2 / 2^*(s)}}\\
&\geq \inf _{Y\setminus\{0\}} \frac{\int_{\mathbb{R}^N} \left(|\nabla w|^2+\left(\frac{y^2}{16}+\frac{N}{4}\right)w^2\right)dy}
{\left(\int_{\mathbb{R}^N} |w|^{2^*(s)}\frac{1}{|y|^s}dy\right)^{2 / 2^*(s)}}\geq S_0.\\
\end{aligned}
\end{equation*}
\end{proof}

\begin{lemma}\label{l45}
 For $N \geqslant 4$ and $\alpha \in (N / 4, N / 2)$,  one has $0<S_{K,\alpha}<S_0 \leqslant S_{K,0}$.
\end{lemma}

\begin{proof}
 Firstly, we  choose $\varphi \in C_c^1\left(\mathbb{R}^{N}\right)$, such that $0\leq \varphi(y) \leq 1$ and
\begin{equation*}
\varphi(y)=\begin{cases}
1 & \text { if } |y| \leqslant 1, \\
0 & \text { if } |y| \geqslant 2.
\end{cases}
\end{equation*}
For $\varepsilon>0$, let
$$
V_{\varepsilon}(y)=\left(\frac{1}{\varepsilon+|y|^{2-s}}\right)^{\frac{N-2}{2-s}}
$$
and $u_{\varepsilon}:=K^{-1 / 2}  \varphi  V_{\varepsilon}$. We prove this lemma by considering three cases.

 (1) We first consider the case $N \geqslant 6$. We claim that for some constant $C>0$ we have:

$$
Q_{\alpha}\left(u_{\varepsilon}\right)=
\begin{cases}
S_0-C(\alpha-(N / 4)) \varepsilon+O(\varepsilon), & \text { if } N \geqslant 7 \\
 S_0-C(\alpha-(N / 4)) \varepsilon|\log \varepsilon|+O(\varepsilon) & \text { if } N=6 .
 \end{cases}
$$
In order to do
this, we  compute
$$
\begin{aligned}
\int_{\mathbb{R}^N}\left|\nabla u_{\varepsilon}\right|^2 K(y) dy=
& \int_{\mathbb{R}^N}\left(\left|\nabla V_{\varepsilon}\right|^2-\frac{1}{2} y \cdot \nabla V_{\varepsilon} \cdot V_{\varepsilon}+\frac{1}{16}|y|^2 V_{\varepsilon}^2\right) \varphi^2dy \\
& +2 \int_{\mathbb{R}^N} \varphi \nabla \varphi \cdot\left(\nabla V_{\varepsilon}-\frac{1}{4} y V_{\varepsilon}\right) V_{\varepsilon}dy +\int_{\mathbb{R}^N}|\nabla \varphi|^2\left|V_{\varepsilon}\right|^2 dy\\
= & T_1+T_2+T_3 .
\end{aligned}
$$
Beginning with the last terms, one can see that $T_2+T_3=O(1)$. Then
\begin{equation*}
\begin{aligned}
T_1&=(N-2)^2 \int_{\mathbb{R}^N} \frac{|y|^{2-2 s}\varphi^2}{\left(\varepsilon+|y|^{2-s}\right)^{\frac{2(N-s)}{2-s}}}dy
+\frac{N-2}{2} \int_{\mathbb{R}^N} \frac{|y|^{2-s}\varphi^2}{\left(\varepsilon+|y|^{2-s}\right)^{\frac{2 N-s-2}{2-s}}}dy\\
&\quad +\frac{1}{16} \int_{\mathbb{R}^N} \frac{|y|^2\varphi^2}{\left(\varepsilon+|y|^{2-s}\right)^{\frac{2(N-2)}{2-s}}}dy.
\end{aligned}
\end{equation*}
Performing  the scaling transformation $y=\varepsilon^{\frac{1}{2-s}} z $, we have
\begin{align*}
&(N-2)^2 \int_{\mathbb{R}^N} \frac{|y|^{2-2 s} \varphi^2}{\left(\varepsilon+|y|^{2-s}\right)^{\frac{2(N-s)}{2-s}}} d y\\
&=(N-2)^2\varepsilon^{\frac{N-2}{s-2}} \int_{\mathbb{R}^N} \frac{|z|^{2-2 s} \left[\varphi\left(\varepsilon^{\frac{1}{2-s}} z\right)\right]^2}{\left(1+|z|^{2-s}\right)^{\frac{2(N-s)}{2-s}}} d z\\
&=(N-2)^2 \varepsilon^{\frac{N-2}{s-2}}\left[\int_{\left|\varepsilon^\frac{1}{2-s} z\right|\leq 1} \frac{|z|^{2-2 s} \left[\varphi\left(\varepsilon^{\frac{1}{2-s}} z\right)\right]^2}{\left(1+|z|^{2-s}\right)^{\frac{2(N-s)}{2-s}}} d z
+\int_{1\leq \left|\varepsilon^\frac{1}{2-s} z\right|\leq 2} \frac{|z|^{2-2 s} \left[\varphi\left(\varepsilon^{\frac{1}{2-s}} z\right)\right]^2}{\left(1+|z|^{2-s}\right)^{\frac{2(N-s)}{2-s}}} d z\right]\\
&=\varepsilon^{\frac{N-2}{s-2}} \left[A_1-(N-2)^2\int_{\left|\varepsilon^\frac{1}{2-s} z\right|\geq 1} \frac{|z|^{2-2 s} \left[\varphi\left(\varepsilon^{\frac{1}{2-s}} z\right)\right]^2}{\left(1+|z|^{2-s}\right)^{\frac{2(N-s)}{2-s}}} d z
+(N-2)^2\int_{1\leq \left|\varepsilon^\frac{1}{2-s} z\right|\leq 2} \frac{|z|^{2-2 s} \left[\varphi\left(\varepsilon^{\frac{1}{2-s}} z\right)\right]^2}{\left(1+|z|^{2-s}\right)^{\frac{2(N-s)}{2-s}}} d z\right]\\
&=\varepsilon^{\frac{N-2}{s-2}}\left(A_1+O\left(\varepsilon^{\frac{2}{2-s}}\right)\right) \quad(\text { for } N \geqslant 3),
\end{align*}
where
\begin{equation*}
A_1:=(N-2)^2 \int_{\mathbb{R}^N} \frac{|y|^{2-2 s} }{\left(1+|y|^{2-s}\right)^{\frac{2(N-s)}{2-s}}} d y
 \quad(\text { for } N \geqslant 3).
\end{equation*}
Similarly, we can obtain
\begin{equation*}
\frac{N-2}{2} \int_{\mathbb{R}^N} \frac{|y|^{2-s} \varphi^2}{\left(\varepsilon+|y|^{2-s}\right)^{\frac{2 N-s-2}{2-s}}} d y=\varepsilon^{\frac{N-4}{s-2}}\left(A_2+O\left(\varepsilon^{\frac{2}{2-s}}\right)\right),
\end{equation*}
where, $$
A_2:=\frac{N-2}{2} \int_{\mathbb{R}^N} \frac{|y|^{2-s}}{\left(1+|y|^{2-s}\right)^{\frac{2 N-s-2}{2-s}}} d y, \quad(\text { for } N \geqslant 5),
$$
and
\begin{equation*}
\frac{1}{16} \int_{\mathbb{R}^N} \frac{|y|^2 \varphi^2}{\left(\varepsilon+|y|^{2-s}\right)^{\frac{2(N-2)}{2-s}}} d y =\varepsilon^{\frac{N-6}{s-2}}\left(A_3+O\left(\varepsilon^{\frac{2}{2-s}}\right)\right),
\end{equation*}
where $$
A_3:=\frac{1}{16} \int_{\mathbb{R}^N} \frac{|y|^2}{\left(1+|y|^{2-s}\right)^{\frac{2(N-2)}{2-s}}} d y,  (\text { for } N \geqslant 7).
$$
When $N = 6$ one has
\begin{equation*}
 \begin{aligned}
\frac{1}{16} \int_{|y|\leq 1} \frac{|y|^2}{\left(\varepsilon+|y|^{2-s}\right)^{\frac{8}{2-s}}} d y
&\leq \frac{1}{16} \int_{|y|\leq 1} \frac{\varphi^2|y|^2}{\left(\varepsilon+|y|^{2-s}\right)^{\frac{8}{2-s}}} d y\\
&\leq \frac{1}{16} \int_{|y|\leq 2} \frac{|y|^2}{\left(\varepsilon+|y|^{2-s}\right)^{\frac{8}{2-s}}} d y
\end{aligned}
\end{equation*}
and $\omega_5$ being the area of the sphere $S^5 \subset \mathbb{R}^6$
\begin{equation*}
\frac{1}{16} \int_{|x|\leq R} \frac{|y|^2}{\left(\varepsilon+|y|^{2-s}\right)^{\frac{8}{2-s}}} d y\leq  \frac{1}{16(2-s)}\omega_5|\log \varepsilon|+O(1) .
\end{equation*}
In fact, let $\gamma=\frac{1}{2-s}, t=\varepsilon+r^{2-s}$, we have
\begin{align}\label{log}
 \int_{|x| \leq R} \frac{|y|^2}{\left(\varepsilon+|y|^{2-s}\right)^{\frac{8}{2-s}}} d y
&= \omega_{5}\int_0^R \frac{r^{7}}{\left(\varepsilon+r^{2-s}\right)^{\frac{8}{2-s}}} d r\nonumber\\
&=\gamma\omega_{5}\int_{\varepsilon}^{\varepsilon+R^{\frac{1}{\gamma}}} \frac{(t-\varepsilon)^{7\gamma}}{t^{8\gamma}} (t-\varepsilon)^{\gamma-1}d t\nonumber\\
&=\gamma\omega_{5}\int_{\varepsilon}^{\varepsilon+R^{\frac{1}{\gamma}}} \frac{(t-\varepsilon)^{8\gamma-1}}{t^{8\gamma}} d t\nonumber\\
&=\gamma\omega_{5}\int_{\varepsilon}^{\varepsilon+R^{\frac{1}{\gamma}}} \frac{\sum_{i=0}C_{8\gamma-1}^it^{8\gamma-1-i}(-\varepsilon)^i}{t^{8\gamma}} d t\nonumber\\
&=\gamma\omega_{5}\int_{\varepsilon}^{\varepsilon+R^{\frac{1}{\gamma}}} \left[\frac{1}{t}+\frac{\sum_{i=1}
C_{8\gamma-1}^it^{8\gamma-1-i}(-\varepsilon)^i}{t^{8\gamma}}\right] d t\nonumber\\
&=\gamma\omega_{5}\int_{\varepsilon}^{\varepsilon+R^{\frac{1}{\gamma}}} \frac{1}{t} d t+O(1)\nonumber\\
&=\gamma\omega_{5}\left(\log(\varepsilon+R^{\frac{1}{\gamma}})-\log(\varepsilon)\right)+O(1)\nonumber\\
&=\gamma\omega_{5}\left(\log \left(1+\frac{R^{\frac{1}{\gamma}}}{\varepsilon}\right)\right)+O(1)\nonumber\\
&=\gamma\omega_{5}\left(\log \left(\frac{R^{\frac{1}{\gamma}}}{\varepsilon}\left(1+\frac{\varepsilon}{R^{\frac{1}{\gamma}}}\right)\right)\right)+O(1)\nonumber\\
&=\frac{1}{2-s}\omega_{5}\left(\log R^{\frac{1}{\gamma}}-\log \varepsilon+\log \left(1+\frac{\varepsilon}{R^{\frac{1}{\gamma}}}\right)\right)+O(1)\nonumber\\
&=\frac{1}{2-s}\omega_{5}|\log \varepsilon|+O(1).
\end{align}
Finally, one has
\begin{equation*}
\int_{\mathbb{R}^N}\left|\nabla u_{\varepsilon}\right|^2 K dy=\left\{\begin{array}{l}
\varepsilon^{\frac{N-2}{s-2}}\left(A_1+\varepsilon^{\frac{2}{2-s}} A_2+\varepsilon^{\frac{4}{2-s}} A_3+O(\varepsilon^{\frac{2}{2-s}})\right)+O(1), \text { for } N \geqslant 7, \\
\varepsilon^{\frac{4}{s-2}}\left(A_1+\varepsilon^{\frac{2}{2-s}} A_2+\frac{1}{2-s}\omega_5 \cdot \varepsilon^{\frac{-4}{s-2}}|\log \varepsilon|+O\left(\varepsilon^{\frac{4}{2-s}}\right)\right)+O(1), \text { for } N=6.
\end{array}\right.
\end{equation*}

We  estimate the term $\alpha \int_{\mathbb{R}^N}\left|u_{\varepsilon}\right|^2 K(y)d y$.

Performing  the scaling transformation $y=\varepsilon^{\frac{1}{2-s}} z $,
One has
\begin{equation*}%\label{}
\begin{aligned}
\alpha\int_{\mathbb{R}^N}\left|u_{\varepsilon}\right|^2 K(y)d y
  &=\alpha\int_{\mathbb{R}^N} \varphi^2\left(\frac{1}{\varepsilon+|y|^{2-s}}\right)^{\frac{2N-4}{2-s}}dy\\
&=\varepsilon^{\frac{4-N}{2-s}}\left(\alpha A_4+O\left(\varepsilon^{\frac{2}{2-s}}\right)\right), \quad N \geqslant 5,
\end{aligned}
\end{equation*}
where,  \begin{equation*}
A_4:=\int_{\mathbb{R}^N} \frac{1}{\left(1+|y|^{2-s}\right)^{\frac{2 N-4}{2-s}}}d y, \quad N \geqslant 5.
\end{equation*}

We  estimate the term

\begin{equation*}
\begin{aligned}
\int_{\mathbb{R}^N}\left|u_{\varepsilon}\right|^{2^{\ast}(s)} K(y)\frac{1}{|y|^s}d y
&=\int_{\mathbb{R}^N}(K(y))^{\frac{s-2}{N-2}}\varphi^{2^{\ast}(s)}\left[\frac{1}{\varepsilon+|y|^{2-s}}\right]^{\frac{2N-2s}{2-s}}\frac{1}{|y|^s}dy,\\
&=\varepsilon^{\frac{s-N}{2-s}}\left(A_0+O(\varepsilon)\right) \quad(\text { for } N \geqslant 3),
\end{aligned}
\end{equation*}
where $$
A_0:=\int_{\mathbb{R}^N} \frac{1}{\left(1+|y|^{2-s}\right)^{\frac{2 N-2 s}{2-s}}|y|^s}d y.
$$
Finally we get
\begin{equation*}
Q_{K,\alpha}(u_{\varepsilon})=\left\{\begin{array}{l}
A_0^{\frac{2-N}{N-s}}\left(A_1+\varepsilon^{\frac{2}{2-s}} \left(A_2-\alpha A_4\right)+\varepsilon^{\frac{4}{2-s}} A_3+O(\varepsilon^{\frac{2}{2-s}})\right), \text { for } N \geqslant 7, \\
A_0^{\frac{-4}{6-s}}\left(A_1+\varepsilon^{\frac{2}{2-s}} \left(A_2-\alpha A_4\right)+\frac{1}{2-s}\omega_5 \cdot \varepsilon^{\frac{-4}{s-2}}|\log \varepsilon|+O\left(\varepsilon^{\frac{2}{2-s}}\right)\right), \text { for } N=6.
\end{array}\right.
\end{equation*}
Then we denote by $\omega_{N-1}$ the area of the sphere $\mathbb{S}^{N-1} \subset \mathbb{R}^N$ and integrate by parts to get
\begin{equation*}
 \begin{aligned}
A_2&=\frac{N-2}{2} \int_{\mathbb{R}^N} \frac{|y|^{2-s}}{\left(1+|y|^{2-s}\right)^{\frac{2 N-s-2}{2-s}}} d y, \\
&=\frac{N-2}{2} \frac{2-s}{4-2N}\frac{1}{2-s}\omega_{N-1}\int_0^{\infty} \frac{r^{1-s+N}}{r^{1-s}} d \left(1+r^{2-s}\right)^{\frac{-2 N+s+2}{2-s}+1}, \\
&=\frac{N}{4}\omega_{N-1} \int_0^{\infty} \frac{r^{N-1}}{\left(1+|r|^{2-s}\right)^{\frac{2 N-4}{2-s}}}d r\\
&=\frac{N}{4}\int_{\mathbb{R}^N} \frac{1}{\left(1+|y|^{2-s}\right)^{\frac{2 N-4}{2-s}}}d y=\frac{N}{4}A_4.\\
 \end{aligned}
\end{equation*}
So it is clear that if $\alpha>A_2 / A_4=N / 4$, then $Q_{K,\alpha}\left(u_{\varepsilon}\right)<A_0^{\frac{2-N}{N-s}}A_1=S_{0}$ for all $N \geqslant 6$, and $\varepsilon$ small enough.

(2) Suppose $N=5$. In this case, the only term which changes in $T_1$ is
\begin{equation*}
\begin{aligned}
\frac{1}{16} \int_{\mathbb{R}^5} \frac{|y|^2\varphi^2}{\left(\varepsilon+|y|^{2-s}\right)^{\frac{6}{2-s}}}dy.
\end{aligned}
\end{equation*}
We compute
\begin{equation*}
\frac{1}{16} \int_{|y| \leq 2} \frac{\varphi^2|y|^2}{\left(\varepsilon+|y|^{2-s}\right)^{\frac{6}{2-s}}} d y\leq
\frac{1}{16} \int_{|y| \leq 2} \frac{|y|^2}{\left(|y|^{2-s}\right)^{\frac{6}{2-s}}} d y=\frac{1}{16} \int_{|y| \leq 2} \frac{|y|^2}{|y|^{6}} d y=O(1).
\end{equation*}
So,
$$
Q_{K,\alpha}\left(u_{\varepsilon}\right)=A_0^{\frac{-3}{5-s}}\left[A_1+\varepsilon^{\frac{2}{2-s}}\left(A_2-\alpha A_4\right)+O\left(\varepsilon^{\frac{2}{2-s}}\right)\right]
$$
%$$A_0^{\frac{-3}{5-s}}\left(A_1+\varepsilon^{\frac{2}{2-s}} \left(A_2-\alpha A_4\right)+\varepsilon^{\frac{4}{2-s}} A_3+O(\varepsilon^{\frac{2}{2-s}})\right)$$
and hence, here also, $\alpha>N / 4$ implies $S_{K,\alpha}<S_0$.

(3)If $N = 4$,  we have
\begin{equation}\label{sk13}
\begin{aligned}
T_1&=4 \int_{\mathbb{R}^4} \frac{|y|^{2-2 s}\varphi^2}{\left(\varepsilon+|y|^{2-s}\right)^{\frac{2(4-s)}{2-s}}}dy
+ \int_{\mathbb{R}^4} \frac{|y|^{2-s}\varphi^2}{\left(\varepsilon+|y|^{2-s}\right)^{\frac{6-s}{2-s}}}dy\\
&\quad +\frac{1}{16} \int_{\mathbb{R}^4} \frac{|y|^2\varphi^2}{\left(\varepsilon+|y|^{2-s}\right)^{\frac{4}{2-s}}}dy.
\end{aligned}
\end{equation}
For the first term on the right-hand side of \eqref{sk13}, we deduce
\begin{equation*}
4 \int_{\mathbb{R}^4} \frac{|y|^{2-2 s} \varphi^2}{\left(\varepsilon+|y|^{2-s}\right)^{\frac{2(4-s)}{2-s}}} d y
=\varepsilon^{\frac{-2}{2-s}}\left(A_1+O(\varepsilon)\right),
\end{equation*}
where $$A_1=4 \int_{\mathbb{R}^4} \frac{|y|^{2-2 s} }{\left(1+|y|^{2-s}\right)^{\frac{2(4-s)}{2-s}}} d y.$$

The second term is bounded as follow
\begin{equation*}
\int_{|y|\leq 1} \frac{|y|^{2-s} \varphi^2}{\left(\varepsilon+|y|^{2-s}\right)^{\frac{6-s}{2-s}}} d y
\leq \int_{|y|\leq 1} \frac{|y|^{2-s} \varphi^2}{\left(\varepsilon+|y|^{2-s}\right)^{\frac{6-s}{2-s}}} d y
\leq  \int_{|y|\leq 2} \frac{|y|^{2-s} }{\left(\varepsilon+|y|^{2-s}\right)^{\frac{6-s}{2-s}}} d y,
\end{equation*}
and similar to \eqref{log}, we  infer
\begin{equation*}
 \int_{|y|\leq 2} \frac{|y|^{2-s} }{\left(\varepsilon+|y|^{2-s}\right)^{\frac{6-s}{2-s}}} d y
 =\omega_3\int_0^2\frac{r^{5-s} }{\left(\varepsilon+r^{2-s}\right)^{\frac{6-s}{2-s}}} d r= \frac{1}{2-s} \omega_3|\log \varepsilon|+O(1),
\end{equation*}
where $\omega_3$ is the area of the sphere $S^3$ of $\mathbb{R}^4$. For the last term can estimate

\begin{equation*}
\frac{1}{16} \int_{{|y|\leq 1}} \frac{|y|^2 }{\left(\varepsilon+|y|^{2-s}\right)^{\frac{4}{2-s}}} d y
\leq \frac{1}{16} \int_{{|y|\leq 2}} \frac{\varphi^2|y|^2 }{\left(\varepsilon+|y|^{2-s}\right)^{\frac{4}{2-s}}} d y
\leq  \frac{1}{16} \int_{{|y|\leq 2}} \frac{|y|^2 }{\left(\varepsilon+|y|^{2-s}\right)^{\frac{4}{2-s}}} d y
\end{equation*}
and
\begin{equation}
\frac{1}{16} \int_{|y| \leq 2} \frac{|y|^2}{\left(\varepsilon+|y|^{2-s}\right)^{\frac{4}{2-s}}} d y
\leq
\frac{1}{16} \int_{|y| \leq 2} \frac{|y|^2}{\left(\varepsilon^2+|y|^{4}\right)}d y=O(1).
\end{equation}

On the other hand, we find that

$$\alpha\int_{\mathbb{R}^4}\left|u_{\varepsilon}\right|^2 K(y)d y
=\alpha \int_{\mathbb{R}^4} \varphi^2\left(\frac{1}{\varepsilon+|y|^{2-s}}\right)^{\frac{4}{2-s}} d y,$$
and
\begin{equation*}
 \alpha \int_{|y|\leq 1} \left(\frac{1}{\varepsilon+|y|^{2-s}}\right)^{\frac{4}{2-s}} d y
 \leq \alpha \int_{|y|\leq 2} \varphi^2\left(\frac{1}{\varepsilon+|y|^{2-s}}\right)^{\frac{4}{2-s}} d y
 \leq \alpha \int_{|y|\leq 2} \left(\frac{1}{\varepsilon+|y|^{2-s}}\right)^{\frac{4}{2-s}} d y.
\end{equation*}
And similar to \eqref{log}, we  infer
 $$\alpha \int_{|y|\leq 2} \left(\frac{1}{\varepsilon+|y|^{2-s}}\right)^{\frac{4}{2-s}} d y
= \frac{\alpha}{2-s} \omega_3|\log \varepsilon|+O(1).
$$
Therefore, we get
$$
Q_{K,\alpha}(u_{\varepsilon})\leq A_0^{\frac{-2}{4-s}}\left(A_1+\frac{\omega_3}{2-s}(1-\alpha) \varepsilon|\log \varepsilon|+O(\varepsilon)\right)
$$
and hence $\alpha>1=N / 4$ implies $S_{K,\alpha}<S_0$.
\end{proof}

\begin{lemma}\label{l46}
Let $N=3$. Then,  if $\alpha \in( 1,3 / 2)$, one has
$$
0<S_{K,\alpha}<S_0 \leqslant S_{K,0}.
$$
\end{lemma}

\begin{proof}
For $\varepsilon>0$, let
$$
\varphi=(K(y))^{-\frac{1}{2}}=e^{-\frac{y^2}{8}},\ \ U_{\varepsilon}(y)=c_{3,s}\left(\frac{\varepsilon}{\varepsilon^2+|y|^{2-s}}\right)^{\frac{3-2}{2-s}}
$$
and $u_{\varepsilon}:=  \varphi^2  U_{\varepsilon}$. Then we calculate
\begin{align*}
\int_{\mathbb{R}^3} K|\nabla u_{\varepsilon}|^2 dy& =\int_{\mathbb{R}^3} \varphi^{-2}\left|\nabla\left(\varphi^2 U_{\varepsilon}\right)\right|^2 dy\\
& =\int_{\mathbb{R}^3}|2 U_{\varepsilon} \nabla \varphi+\varphi \nabla U_{\varepsilon}|^2 dy\\
& =\int_{\mathbb{R}^3} 4 U_{\varepsilon}^2|\nabla \varphi|^2+2 \varphi U_{\varepsilon}(\nabla \varphi \cdot \nabla U_{\varepsilon})+\operatorname{div}\left(\varphi^2 U_{\varepsilon} \nabla U_{\varepsilon}\right)-\varphi^2 U_{\varepsilon} \Delta U_{\varepsilon} dy\\
& =\int_{\mathbb{R}^3} 4 U_{\varepsilon}^2|\nabla \varphi|^2+\frac{1}{2}\left(\nabla \varphi^2 \cdot \nabla U_{\varepsilon}^2\right)+\varphi^2 U_{\varepsilon}^{2^*(s)}\frac{1}{|y|^s}+\operatorname{div}\left(\varphi^2 U_{\varepsilon} \nabla U_{\varepsilon}\right) dy\\
& =\int_{\mathbb{R}^3} 4 U_{\varepsilon}^2|\nabla \varphi|^2-\frac{1}{2} U_{\varepsilon}^2 \Delta \varphi^2+\varphi^2 U_{\varepsilon}^{2^*(s)}\frac{1}{|y|^s}
+\operatorname{div}\left(\varphi^2 U_{\varepsilon} \nabla U_{\varepsilon}+\frac{1}{2} U_{\varepsilon}^2 \nabla \varphi^2\right)dy.
\end{align*}
And  by Gauss divergence theorem, we get
$$\int_{\mathbb{R}^3}\operatorname{div}\left(\varphi^2 U_{\varepsilon} \nabla U_{\varepsilon}+\frac{1}{2} U_{\varepsilon}^2 \nabla \varphi^2\right)dy=0$$
and
$$
\int_{\mathbb{R}^3} K\left|\nabla u_{\varepsilon}\right|^2 d y
=\int_{\mathbb{R}^3} \left(4 U_{\varepsilon}^2|\nabla \varphi|^2-\frac{1}{2} U_{\varepsilon}^2 \Delta \varphi^2+\varphi^2 U_{\varepsilon}^{2^*(s)} \frac{1}{|y|^s}\right)dy.
$$
Using the explicit form of $\varphi$, a direct calculation yields
$$
4|\nabla \varphi|^2-\frac{1}{2} \Delta \varphi^2=\frac{3}{4}  \varphi^2+\frac{1}{8} |y|^{2} \varphi^2.
$$
Thus,  we can obtain
$$
\int_{\mathbb{R}^3} 4 U_{\varepsilon}^2|\nabla \varphi|^2-\frac{1}{2} U_{\varepsilon}^2 \Delta \varphi^2dy
=\int_{\mathbb{R}^3} U_{\varepsilon}^2\left(\frac{3}{4}  \varphi^2+\frac{1}{8} |y|^{2} \varphi^2\right)dy.
$$
We now prove  that
\begin{equation}\label{sk17}
 \int_{\mathbb{R}^3} U_{\varepsilon}^2\left(\frac{3}{4} \varphi^2+\frac{1}{8}|y|^2 \varphi^2\right) d y
=c_{3,s} \varepsilon^{\frac{2}{2-s}} \int_{\mathbb{R}^3} |y|^{-2} \left(\frac{3}{4}  \varphi^2+\frac{1}{8} |y|^{2} \varphi^2\right)dy+o\left(\varepsilon\right).
\end{equation}
Hence, one has
\begin{equation}\label{sk16}
  \int_{\mathbb{R}^3} K\left|\nabla u_{\varepsilon}\right|^2 d y=\int_{\mathbb{R}^3}\varphi^2 U_{\varepsilon}^{2^*(s)} \frac{1}{|y|^s} d y
+ c_{3,s}\varepsilon^{\frac{2}{2-s}} \int_{\mathbb{R}^3}|y|^{-2} \left(\frac{3}{4} \varphi^2+\frac{1}{8}|y|^2 \varphi^2\right) d y+o\left(\varepsilon\right).
\end{equation}
In fact,
using Young inequality
$$
\frac{a^p}{p}+\frac{b^q}{q} \geqslant a b, \quad \text { with } a=\left(p\varepsilon^2\right)^{1 / p},  b=\left(q |y|^{2-s}\right)^{1 / q}, q>\max\{1,2-s\},\frac{1}{p}+\frac{1}{q}=1,
$$
we have
$$
\left(\varepsilon^2+|y|^{2-s}\right)\geq \left(p\varepsilon\right)^{1 / p}\left(q |y|^{2-s}\right)^{1 / q}.
$$
And using the inequality  $$
(c+\varepsilon)^{\beta}-c^{\beta}=\int_c^{c+\varepsilon} \beta x^{\beta-1} d x<\int_c^{c+\varepsilon} \beta(c+\varepsilon)^{\beta-1} d x=\beta \varepsilon(c+\varepsilon)^{\beta-1},(\beta>1,c>0).
$$
we get
\begin{align*}
\frac{\left(\varepsilon^2+|y|^{2-s}\right)^{\frac{2}{2-s}}-|y|^2}{|y|^2\left(\varepsilon^2+|y|^{2-s}\right)^{\frac{2}{2-s}}}
&\leq \frac{\left(\varepsilon^2+|y|^{2-s}\right)^{\frac{2}{2-s}}-\left(|y|^{2-s}\right)^{\frac{2}{2-s}}}{|y|^2\left(\varepsilon^2+|y|^{2-s}\right)^{\frac{2}{2-s}}}\\
&\leq \frac{\frac{2\varepsilon^2}{2-s}\left(\varepsilon^2+|y|^{2-s}\right)^{\frac{s}{2-s}}}{|y|^2\left(\varepsilon^2+|y|^{2-s}\right)^{\frac{2}{2-s}}}\\
&\leq  \frac{\frac{2\varepsilon^2}{2-s}}{|y|^2(\varepsilon^2+|y|^{2-s})}\\
&\leq  \frac{\frac{2\varepsilon^2}{2-s}}{|y|^2(\left(p\varepsilon\right)^{1 / p}\left(q |y|^{2-s}\right)^{1 / q})}.
\end{align*}
According to the above inequality, we infer
\begin{align}\label{sk18}
&\left|\int_{\mathbb{R}^3}|y|^{-2}\left(\frac{3}{4} \varphi^2+\frac{1}{8}|y|^2 \varphi^2\right) d y
-\int_{\mathbb{R}^3}\left(\frac{1}{\varepsilon^2+|y|^{2-s}}\right)^{\frac{2}{2-s}}\left(\frac{N}{4} \varphi^2+\frac{1}{8}|y|^2 \varphi^2\right) d y\right|\nonumber\\
&=\int_{\mathbb{R}^3}\frac{\left(\varepsilon^2+|y|^{2-s}\right)^{\frac{2}{2-s}}-|y|^2}{|y|^2\left(\varepsilon^2+|y|^{2-s}\right)^{\frac{2}{2-s}}}\left(\frac{3}{4} \varphi^2+\frac{1}{8}|y|^2 \varphi^2\right) d y\nonumber\\
&\leq \int_{\mathbb{R}^3}\frac{\frac{2\varepsilon^2}{2-s}}{|y|^2(\left(p\varepsilon\right)^{1 / p}\left(q |y|^{2-s}\right)^{1 / q})}\left(\frac{3}{4} \varphi^2+\frac{1}{8}|y|^2 \varphi^2\right) d y.
\end{align}
From $q>\max\{1,2-s\}$, we find that the  integral on the right-hand side of \eqref{sk18} is
convergent. So, \eqref{sk18} and the definition of $U_{\varepsilon}(y)$ imply  that (\ref{sk17}) holds.

By Gauss divergence theorem, we have
$$0=\int_{\mathbb{R}^3}
\operatorname{div}\left(\frac{y}{|y|^{2}} e^{-\frac{1}{4}|y|^2}\right)dy
=\int_{\mathbb{R}^3} e^{-\frac{|y|^2}{4}}\left(\frac{1}{|y|^{2}}-\frac{1}{2}\right)dy, $$
from which we may infer that
$$\int_{\mathbb{R}^3} \frac{\varphi^2}{|y|^{2}}dy=\int_{\mathbb{R}^3} \frac{\varphi^2}{2}dy.$$
Consequently, (\ref{sk16}) turns into
\begin{equation}\label{skkk16}
  \int_{\mathbb{R}^3} K\left|\nabla u_{\varepsilon}\right|^2 d y=\int_{\mathbb{R}^3}\varphi^2 U_{\varepsilon}^{2^*(s)} \frac{1}{|y|^s} d y
+ c_{3,s}\varepsilon^{\frac{2}{2-s}} \int_{\mathbb{R}^3}|y|^{-2} \varphi^2 d y+o\left(\varepsilon\right).
\end{equation}
The same estimate used for \eqref{sk17} gives
\begin{align}\label{sk19}
\int_{\mathbb{R}^3} K\left| u_{\varepsilon}\right|^2 d y&= \int_{\mathbb{R}^3} K \varphi^4 U^2_{\varepsilon}dy= \int_{\mathbb{R}^3}  \varphi^2 U^2_{\varepsilon}dy\nonumber\\
&= c_{3,s}\varepsilon^{\frac{2}{2-s}}\int_{\mathbb{R}^3} \varphi^2 \left(\frac{1}{\varepsilon^2+|y|^{2-s}}\right)^{\frac{2}{2-s}}dy\nonumber\\
&= c_{3,s}\varepsilon^{\frac{2}{2-s}}\int_{\mathbb{R}^3} \varphi^2 |y|^{-2}dy+o(\varepsilon).
\end{align}
Then \eqref{skkk16} and \eqref{sk19} imply that
$$
\int_{\mathbb{R}^3} K(y)\left|\nabla u_{\varepsilon}\right|^2 d y-\alpha\int_{\mathbb{R}^3} K\left| u_{\varepsilon}\right|^2 d y
=\int_{\mathbb{R}^3} \varphi^2 U_{\varepsilon}^{2^*(s)} \frac{1}{|y|^s} d y+(1-\alpha)c_{3,s}\varepsilon^{\frac{2}{2-s}}\int_{\mathbb{R}^3} \varphi^2 |y|^{-2}dy+o(\varepsilon).
$$
It is easy to see that
$$
\int_{\mathbb{R}^3}  u_{\varepsilon}^{2^*(s)} K(y)\frac{1}{|y|^s}dy =\int_{\mathbb{R}^3} \varphi^{12-4s} K(y)U_{\varepsilon}^{2^*(s)}\frac{1}{|y|^s}dy.
$$
We prove that the following estimate holds
\begin{equation}\label{sk21}
 \int_{\mathbb{R}^3} \varphi^{12-4 s} U_{\varepsilon}^{2^*(s)} \frac{1}{|y|^s} K(y)d y
=\int_{\mathbb{R}^3}  U_{\varepsilon}^{2^*(s)} \frac{1}{|y|^s} d y
+\int_{\mathbb{R}^3}\left(\varphi^{12-4 s}-1\right) (c_{3,s})^{2^*(s)} \left(\frac{\varepsilon}{\varepsilon^2+|y|^{2-s}}\right)^{\frac{6-s}{2-s}}\frac{1}{|y|^s}dy.
\end{equation}
In fact, we choose $p,q>1$ and $$\frac{6-2s}{5-s}<q<\frac{6-2s}{3-s},\frac{1}{p}+\frac{1}{q}=1,$$
then using Young inequality  again,
$$
\frac{a^p}{p}+\frac{b^q}{q} \geqslant a b, \quad \text { with } a=\left(p\varepsilon^2\right)^{1 / p},  b=\left(q |y|^{2-s}\right)^{1 / q},
$$
we obtain
\begin{align}\label{sk20}
&\int_{\mathbb{R}^3}\left(\varphi^{12-4 s}-1\right) (c_{3,s})^{2^*(s)} \left(\frac{\varepsilon}{\varepsilon^2+|y|^{2-s}}\right)^{\frac{6-2s}{2-s}}\frac{1}{|y|^s}dy\nonumber\\
&\leq\int_{\mathbb{R}^3}\left(\varphi^{12-4 s}-1\right) (c_{3,s})^{2^*(s)} \left(\frac{\varepsilon}{(p \varepsilon)^{1 / p}\left(q|y|^{2-s}\right)^{1 / q}}\right)^{\frac{6-2s}{2-s}}\frac{1}{|y|^s}dy.
\end{align}
It is easy to see that when $|y|$ is sufficiently small, $1-\varphi^r(r>1)$ is equivalent to $|y|^2$.  And
$$
1<\frac{6-2s}{5-s}<q<\frac{6-2s}{3-s},
$$
we can deduce that $$\frac{6-2s}{q}+s>3,\ \frac{6-2s}{q}+s-2<3,$$
which implies that the right-hand side of (\ref{sk20}) is integrable. So, \eqref{sk21} holds.

Then, the definition of $Q_{K,\alpha}$ gives

$$Q_{K,\alpha}=\frac{\int_{\mathbb{R}^3} \varphi^2U_{\varepsilon}^{2^*(s)} \frac{1}{|y|^s}  d y+(1-\alpha) c_{3, s} \varepsilon^{\frac{2}{2-s}} \int_{\mathbb{R}^3} \varphi^2|y|^{-2} d y+o(\varepsilon)}{\int_{\mathbb{R}^3} U_{\varepsilon}^{2^*(s)} \frac{1}{|y|^s} K(y) d y}<S_0.$$
\end{proof}

\begin{proof}[Proof of Theorem \ref{t42}]
It follows immediately from Lemma \ref{l41},Lemma \ref{l44},  Lemma \ref{l45} and Lemma \ref{l46} that Theorem \ref{t42} is valid.
\end{proof}

\section{Another proof of Theorem \ref{t42}}
In this section, we give another proof of Theorem \ref{t42}.
\begin{equation}\label{e51}
E_K(v)=\frac{1}{2} \int_{\mathbb{R}^N}|\nabla v|^2 K(y) d y-\frac{\alpha}{2} \int_{\mathbb{R}^N}|v|^2 K(y) d y
-\frac{1}{2^{\ast}(s)} \int_{\mathbb{R}^N} \frac{1}{|y|^s}|v|^{2^{\ast}(s)} K(y) d y.
\end{equation}

\begin{lemma}\label{l51}
 Assume that the functional $E_K(v)$ is defined as in (\ref{e51}),  $N\geq 3$. Then for every $\alpha\in \mathbb{R}$,
  every sequence $\left\{v_n\right\}$ in $X$ such that, as $n \rightarrow \infty$,

\begin{equation}\label{e58}
E_K\left(v_n\right) \rightarrow \beta<\frac{2-s}{2(N-s)} \left(S_{K,0}\right)^{\frac{N-s}{2-s}},
\quad \langle E^{\prime}_{K}(v_n),v_n\rangle \rightarrow 0,
\end{equation}
is relatively compact in $X$.
\end{lemma}

\begin{proof}
 It is easy to check that $\left\{v_n\right\}$ is bounded in $X$. In fact,

$$
\begin{aligned}
&\frac{2-s}{2(N-s)} \left(S_{K,0}\right)^{\frac{N-s}{2-s}}+o(1)\left(1+\left\|\nabla v_n\right\|_{K,2}\right)
 \geq E_K\left(v_n\right)-\frac{1}{2}\left\langle E_K^{\prime}(v_n), v_n\right\rangle \\
& \quad=\left(\frac{1}{2}-\frac{1}{2^{\ast}(s)}\right) \int_{\mathbb{R}^N} \frac{1}{|y|^s}|v_n|^{2^{\ast}(s)} K(y) d y.
\end{aligned}
$$
So by Poincar\'{e} inequality, we have

$$
\begin{aligned}
\left(1-\frac{\alpha}{\lambda_1}\right) \int_{\mathbb{R}^N}|\nabla v_n|^2 K(y) d y &\leq \int_{\mathbb{R}^N}|\nabla v_n|^2 K(y) d y
-\alpha \int_{\mathbb{R}^N}|v_n|^2 K(y) d y\\
&=2 E_K\left(v_n\right)+\frac{2}{2^*(s)} \int_{\mathbb{R}^N} \frac{1}{|y|^s}|v|^{2^*(s)} K(y) d y. \\
& \leq C_2+o(1)\left\|\nabla v_n\right\|_{K,2},
\end{aligned}
$$
where $o(1) \rightarrow 0$ as $n \rightarrow \infty$. Thus it follows that $\left\{v_n\right\}$ is bounded in $X$.
So we can extract a subsequence, still denoted by $\left\{v_n\right\}$, such that, as $n \rightarrow \infty$,
\begin{equation}\label{e53}
\begin{array}{ll}
v_n \rightharpoonup v & \text { weakly in } X, \\
v_n \rightarrow v & \text { in } L_K^2\left(\mathbb{R}^N\right), \\
v_n \rightarrow v & \text { in } L_K^q\left(\mathbb{R}^N,\frac{1}{|y|^s}\right), 2 \leq q<2^{\ast}(s),\\
v_n \rightarrow v & \text { a.e. on } \mathbb{R}^N.
\end{array}
\end{equation}
  In particular, for every $\varphi \in X$, we obtain that, as $n \rightarrow \infty$,
 \begin{equation}\label{e52}
 \begin{aligned}
\left\langle E_K^{\prime}\left(v_n\right), \varphi\right\rangle&=   \int_{ \mathbb{R}^N }\nabla v_n\nabla \varphi
 K(y)d y- \alpha\int_{ \mathbb{R}^N }v_n\varphi K(y)d y- \int_{\mathbb{R}^N }\frac{1}{|y|^s}|v_n|^{2^{\ast}(s)-1}\varphi K(y)dy\\
& \rightarrow \int_{ \mathbb{R}^N }\nabla v\nabla \varphi
 K(y)d y- \alpha\int_{ \mathbb{R}^N }v\varphi K(y)d y- \int_{\mathbb{R}^N }\frac{1}{|y|^s}|v|^{2^{\ast}(s)-1}\varphi K(y)dy\\
 &=\left\langle E_K^{\prime}(v), \varphi\right\rangle.
\end{aligned}
 \end{equation}

Since by hypothesis $E_K^{\prime}\left(v_n\right) \rightarrow 0$,
 we deduce $\left\langle E_K^{\prime}(v), \varphi\right\rangle=0$.
 Thus, $v \in X$ solves the  problem (\ref{e02}).

 Next we show that $v_n \rightarrow v$ strongly in $X$.

 We have by choosing $\varphi=v$ in \eqref{e52},
\begin{equation}
\left\langle E_K^{\prime}(v), v\right\rangle
=\int_{\mathbb{R}^N} |\nabla v|^2  K(y) d y-\alpha \int_{\mathbb{R}^N} v^2 K(y) d y-\int_{\mathbb{R}^N} \frac{1}{|y|^s}|v|^{2^*(s)} K(y) d y=0.
\end{equation}
Hence, we see
\begin{equation}\label{e55}
 E_K(v)=\left(\frac{1}{2}-\frac{1}{2^*(s)}\right)
 \int_{\mathbb{R}^N} \frac{1}{|y|^s}|v|^{2^*(s)} K(y) d y=\frac{2-s}{2N-2s}
 \int_{\mathbb{R}^N} \frac{1}{|y|^s}|v|^{2^*(s)} K(y) d y\geq 0.
\end{equation}
Moreover, Lemma \ref{b2} leads to
\begin{equation*}
\int_{\mathbb{R}^N} \frac{\left|v_n-v\right|^{2^*(s)}}{|y|^s} K(y) d y=\int_{\mathbb{R}^N} \frac{\left|v_n\right|^{2^*(s)}}{|y|^s} K(y) d y-\int_{\mathbb{R}^N} \frac{|v|^{2^*(s)}}{|y|^s} K(y) d y+o(1)
\end{equation*}
Notice that, developing the square and by weak convergence,
\begin{equation*}
\int_{\mathbb{R}^N}\left|\nabla v_n\right|^2  d y=\int_{\mathbb{R}^N}\left|\nabla v_n-\nabla v\right|^2  d y+\int_{\mathbb{R}^N}|\nabla v|^2  d y+o(1).
\end{equation*}
Thus, according to $v_n \rightarrow v$ in $L_K^2\left(\mathbb{R}^N\right)$ and  the above two formulas,  we have
\begin{equation}\label{e54}
E_K\left(v_n\right)=E_K(v)+\frac{1}{2}\int_{\mathbb{R}^N}\left|\nabla v_n-\nabla v\right|^2 d y-\frac{1}{2^{\ast}(s)}\int_{\mathbb{R}^N} \frac{\left|v_n-v\right|^{2^*(s)}}{|y|^s} K(y) d y+o(1) .
\end{equation}
Furthermore, by \eqref{e53}, we obtain
\begin{align*}
& \int_{\mathbb{R}^N}\left(\left|v_n\right|^{2^*(s)-1} v_n-|v|^{2^*(s)-2} v\right)\left(v_n-v\right) \frac{1}{|y|^s}K(y)d y \\
& =\int_{\mathbb{R}^N}\left(\left|v_n\right|^{2^*(s)}-\left|v_n\right|^{2^*(s)-2} v_n v\right) \frac{1}{|y|^s}K(y)d y+o(1) \\
& =\int_{\mathbb{R}^N}\left(\left|v_n\right|^{2^*(s)}-|v|^{2^*(s)}\right) \frac{1}{|y|^s}K(y)d y+o(1)\\
&=\int_{\mathbb{R}^N}\left(\left|v_n-v\right|^{2^*(s)}\right) \frac{1}{|y|^s}K(y)d y+o(1),
\end{align*}
which gives
\begin{equation*}
\begin{aligned}
o(1) & =\left\langle E_K^{\prime}\left(v_n\right), v_n-v\right\rangle \\
& =\left\langle E_K^{\prime}\left(v_n\right)-E_K^{\prime}(v), v_n-v\right\rangle \\
& =\int_{\mathbb{R}^N}\left|\nabla\left(v_n-v\right)\right|^2 K(y)d y-\int_{\mathbb{R}^N}\left|v_n-v\right|^{2^{\ast}(s)} \frac{1}{|y|^s}K(y)d y+o(1).
\end{aligned}
\end{equation*}
Then according to \eqref{e54} we obtain
\begin{equation*}
E_K\left(v_n\right)-E_K(v)=\frac{1}{2}\int_{\mathbb{R}^N}\left|\nabla v_n-\nabla v\right|^2 d y-\frac{1}{2^{\ast}(s)}\int_{\mathbb{R}^N} \frac{\left|v_n-v\right|^{2^*(s)}}{|y|^s} K(y) d y\rightarrow 0.
\end{equation*}
On the other hand, by \eqref{e54} and since $E_K(v) \geq 0$ (by \eqref{e55}), we see that there is a large $n_0>0$ such that, for $n \geq n_0$,
\begin{equation*}
\begin{aligned}
&\frac{1}{2}\int_{\mathbb{R}^N}\left|\nabla v_n-\nabla v\right|^2 d y-\frac{1}{2^{\ast}(s)}
\int_{\mathbb{R}^N} \frac{\left|v_n-v\right|^{2^*(s)}}{|y|^s} K(y) d y\\
&=E_K\left(v_n\right)-E_K(v)+o(1)\\
&\leq E_K\left(v_n\right)+o(1)\leq \frac{2-s}{2(N-s)} \left(S_{K,0}\right)^{\frac{N-s}{2-s}}.
\end{aligned}
\end{equation*}
Therefore, we have the following inequality
$$
\left\|\nabla v_n-\nabla v\right\|_{K,2}^2<\frac{2-s}{N-s} \left(S_{K,0}\right)^{\frac{N-s}{2-s}}
\leq \left(\frac{\int_{\mathbb{R}^N} K(y)|\nabla v_n-v|^2 d y}{\left(\int_{\mathbb{R}^N} K(y)|v_n-v|^{2 *(s)} \frac{1}{|y|^s} d y\right)^{2 / 2^*(s)}}\right)^{\frac{N-s}{2-s}},
$$
which gives
$$
\frac{\int_{\mathbb{R}^N} K(y)|v_n-v|^{2^*(s)} \frac{1}{|y|^s} d y}{\int_{\mathbb{R}^N} K(y)|\nabla (v_n-v)|^2 d y} \leq c<1
$$
for all $n \geq n_0$. Then we obtain that, as $n \rightarrow \infty$,
\begin{equation}
\begin{aligned}
\left(1-c\right)\left\|\nabla v_n- \nabla v\right\|^2_{K} &
\leq\left\|\nabla v_n- \nabla v\right\|_{K}^2\left(1-\frac{\int_{\mathbb{R}^N}\left|v_n-v\right|^{2^{\ast}(s)}\frac{1}{|y|^s}K(y) dy}{\int_{\mathbb{R}^N}\left|\nabla\left(v_n-v\right)\right|^2 K(y)d y}\right) \\
& \leq \int_{\mathbb{R}^N}\left|\nabla\left(v_n-v\right)\right|^2 K(y)d y-\int_{\mathbb{R}^N}\left|v_n-v\right|^{2^{\ast}(s)} \frac{1}{|y|^s}d y=o(1),
\end{aligned}
\end{equation}
establishing that $v_n \rightarrow v$ strongly in $X$.
\end{proof}
Before giving the second proof of Theorem \ref{t42}, let us recall the Mountain Pass Theorem
which can be  consulted in \cite{cr4}.
\begin{lemma}\label{l52}
 Let $E$ be a real Banach space with its dual space $E^*$ and suppose that $I \in C^1(E, R)$ satisfies the condition
$$
\max \left\{I(0), I\left(u_1\right)\right\} \leq \alpha<\beta \leq \inf _{\left\|u_1\right\|=\rho} I(u),
$$
for some $\beta>\alpha, \rho>0$ and $u_1 \in E$ with $\left\|u_1\right\|>\rho$. Let $c \geq \beta$ be characterized by
$$
c=\inf _{\gamma \in \Gamma} \max _{0 \leq \tau \leq 1} I(\gamma(\tau)),
$$
where $\Gamma=\left\{\gamma \in C([0,1], E) \mid \gamma(0)=0, \gamma(1)=u_1\right\}$ is the set of continuous paths joining 0 and $u_1$. Then, there exists a sequence $\left\{u_n\right\} \subset E$ such that, as $n \rightarrow \infty$,
$$
I\left(u_n\right) \rightarrow c \geq \beta \text { and }\left.I^{\prime}\left(u_n\right)\right|_{E^* \rightarrow 0} .
$$
\end{lemma}

Let
$$
\Sigma=\left\{v \in X \backslash\{0\} \mid\left\langle E_K^{\prime}(v), v\right\rangle=0\right\} .
$$
We define critical values for the functionals as follows:
$$
\begin{aligned}
c^* & =\inf _{v \in \Sigma} E_K(v), \\
c & =\inf _{\gamma \in \Gamma} \sup _{t \in[0,1]} E_K(\gamma(t)), \\
c^{* *} & =\inf _{v \in X) \backslash\{0\}} \sup _{t \geq 0} E_K(t v),
\end{aligned}
$$
where $\Gamma:=\left\{\gamma \in C\left([0,1], X\right) \mid \gamma(0)=0, E_K(\gamma(1))<0\right\}$.
We have the following relations, whose proofs are standard (see \cite[Theorem 4.2]{cr4}).

\begin{lemma}\label{l53}
$c=c^*=c^{* *}$
\end{lemma}

\begin{proof}[Proof of Theorem \ref{t42}]
It is sufficient to prove that $E_K(v)$ satisfies the condition (\ref{e58}).
Considering the functional $E_K(v)$, for every $v\in X$ and $t \geq 0$, we have
$$
E_K(t v)=\frac{A t^2}{2}-\frac{B t^{2^{\ast}(s)}}{2^{\ast}(s)},
$$
where
$$
A=\int_{\mathbb{R}^N}\left(|\nabla v|^2-\alpha|v|^2\right) K(y) d y
$$
and
$$
B=\int_{\mathbb{R}^N} \frac{1}{|y|^s}|v|^{2^*(s)} K(y) d y.
$$
We see that $E_K(t v)$ has its maximum at $t_0=\left(\frac{A}{B}\right)^{1 /\left(2^{\ast}(s)-2\right)}=\left(\frac{A}{B}\right)^{(N-2) /(4-2s)}$. Hence, we obtain

$$
\sup _{t \geq 0} E_K(t v)=\max _{t \geq 0} E_K(t v)=\left.E_K(t v)\right|_{t=t_0}=\frac{2-s}{2(N-s)} A\left(\frac{A}{B}\right)^{\frac{N-2}{2-s}}.
$$
This implies
$$
\begin{aligned}
\inf _{0 \neq v \in X} \sup _{t \geq 0} E_K(t v) & \leq \frac{2-s}{2(N-s)} \left(\inf _{0 \neq v \in X} Q_{K,\alpha}(v)\right)^{\frac{N-s}{2-s}} \\
& <\frac{2-s}{2(N-s)} \left(S_{K,0}\right)^{\frac{N-s}{2-s}}<\frac{2-s}{2(N-s)} \left(S_{0}\right)^{\frac{N-s}{2-s}}.
\end{aligned}
$$
by Lemma \ref{l44},  Lemma \ref{l45} and Lemma \ref{l46}. Then by using Lemma \ref{l51},  Lemma \ref{l52} and  Lemma \ref{l53}, we obtain
$$
c^{* *}=\inf _{0 \neq v \in H_{0, L}^1(\mathcal{C})} \sup _{t \geq 0} I_1(t v)
$$
is a critical value of $E_K$. %Finally we complete the proof of regularity of the solution by XXX
\end{proof}

\end{document}